\documentclass[11pt]{article}

\usepackage[T1]{fontenc}
\usepackage{lmodern}
\usepackage[margin=1in]{geometry}
\usepackage{amsthm,amsmath,amsfonts,amssymb,mathtools}
\usepackage{booktabs,array,multirow}
\usepackage{graphicx,subcaption}
\usepackage{xcolor,enumitem}
\usepackage{microtype}
\usepackage[authoryear,round]{natbib}
\usepackage[colorlinks=true,allcolors=blue]{hyperref}

\theoremstyle{plain}
\newtheorem{theorem}{Theorem}[section]
\newtheorem{proposition}[theorem]{Proposition}
\newtheorem{corollary}[theorem]{Corollary}
\newtheorem{lemma}[theorem]{Lemma}
\theoremstyle{definition}
\newtheorem{assumption}[theorem]{Assumption}
\newtheorem{definition}[theorem]{Definition}
\newtheorem{remark}[theorem]{Remark}

\newcommand{\R}{\mathbb R}
\newcommand{\E}{\mathbb E}
\newcommand{\Pp}{\mathbb P}
\newcommand{\op}{\mathrm{op}}
\newcommand{\tr}{\mathrm{tr}}
\newcommand{\diag}{\mathrm{diag}}
\newcommand{\Risk}{\mathcal R}
\newcommand{\Rbar}{\overline R}
\newcommand{\dPCR}{\mathrm{dPCR}}
\newcommand{\PCR}{\mathrm{PCR}}
\newcommand{\1}{\mathbf 1}
\newcommand{\norm}[1]{\left\lVert#1\right\rVert}
\newcommand{\ip}[2]{\left\langle#1,#2\right\rangle}
\DeclareMathOperator*{\argmin}{arg\,min}

\mathtoolsset{showonlyrefs}
\title{De-floored Principal Component Regression: When Rank Selection Alone Is Insufficient for Prediction}
\author{Peng Zhao\\
\small Department of Applied Economics and Statistics, University of Delaware\\
}
\date{}

\begin{document}
\maketitle

\begin{abstract}
Principal component regression (PCR) regularizes high-dimensional prediction by
choosing a spectral cutoff, but rank selection cannot correct systematic
inflation of the retained empirical eigenvalues.  We study clean Gaussian
random designs in which the aggregate covariance tail creates a nearly scalar
sample-space floor comparable to the predictive head scale.  De-floored
principal component regression (dPCR) retains the cutoff and subtracts an
estimated floor from the retained denominators.  We prove an ordinary-PCR
prediction-risk lower bound uniform over all ranks and a high-probability dPCR
upper bound.  When the floor is sharp and inexpensive to remove in population
prediction risk, the conditional risk of dPCR is asymptotically negligible
relative to that of the best ordinary PCR rank.  An exact risk decomposition
explains the separation: denominator inflation is governed by first spectral
mass, whereas the clean prediction cost of correction is governed by squared
spectral mass.  A same-sample trimmed-mean floor estimate attains the oracle
dPCR upper-bound rate at a prespecified rank, and the separation persists under
approximate predictive alignment when the tail prediction-energy fraction
vanishes. 
Separate pointwise fixed-aspect formulas show that the risk-optimal positive
scalar correction improves rank-$1$ PCR, whereas mean-floor subtraction is
generally not optimal for a broad Marchenko--Pastur bulk.
\end{abstract}

\medskip

\section{Introduction}
\label{sec:intro}

PCR regularizes high-dimensional prediction by choosing how many empirical
spectral directions to retain.  Rank selection does not control a second source
of risk: the empirical eigenvalues in the retained inverse can themselves be
systematically inflated by many weak directions. Existing PCR theory usually places this inflation in one of three regimes: it
is negligible at the retained signal scale, it is interpreted as beneficial
self-induced regularization, or it belongs to a broad
Marchenko--Pastur spectrum \citep{marchenko1967distribution} that is intrinsically non-scalar.  This paper studies
the missing intersection in which the inflation is nonnegligible, sharp,
matched to predictive signal, and inexpensive to remove.
Table~\ref{tab:missing-regime} locates this regime by floor size and sharpness, with
squared tail mass determining whether correction is affordable.

This failure is easiest to see in sample space.  After diagonalizing the
population covariance, write the Gram matrix as $K=K_H+K_T^{(1)}$, where $K_H$
contains the leading population-eigenvalue block and $K_T^{(1)}$ aggregates the
remaining weak spectral coordinates.  We assume that prediction signal is
strongly aligned with this population head.  When the tail has high effective
rank, $K_T^{(1)}$ can be close to $aI_n$ even though no population eigenvalue equals
$a$ \citep{bartlett2020benign,hastie2022surprises}.  A predictive head term with
realized scale $s_j$ is then observed near $s_j+a$.  PCR normalizes it
by $s_j+a$, leaving structural multiplier
\begin{equation}
    q_{\PCR,j}=\frac{s_j}{s_j+a}.
    \label{eq:intro-pcr-multiplier}
\end{equation}
When $s_j\asymp a$, the bias is constant order.  Better rank selection cannot
remove it.

We study the direct correction suggested by this calculation.  For empirical
eigenpairs $(\widehat\mu_i,\widehat v_i)$, dPCR retains the first $r$ components
but divides by $\widehat\mu_i-a_0$ rather than $\widehat\mu_i$.  The correction
level $a_0$ changes the retained inverse weights, while $r$ controls omission
and variance.  A positive
denominator margin keeps the truncated inverse away from its pole.  Thus dPCR is
not a full inverse with a shifted denominator: the hard cutoff is part of the
estimator and part of its stability condition.

PCR prediction has been studied through classical component-selection
criteria, finite-sample comparisons with ridge, minimax spectral
regularization, average-case risk calculations, empirical-versus-oracle
bounds, latent-factor models, and proportional random-matrix limits
\citep{jolliffe1982note,dhillon2013risk,dicker2017kernel,xu2019variables,
hucker2023prediction,bing2021latent,wu2020weighted,brownlees2024pcr,
gedon2024pcr,green2025pcr}.  Complementary work develops implicit and sketched
approximations for large-scale PCR
\citep{frostig2016projection,allenzhu2017fasterpcr,moryosef2019sketching}.
The head--tail proof architecture used below also connects to benign-overfitting
analyses of minimum-norm least squares, alternative interpolants, ridge, and
stochastic-gradient methods
\citep{muthukumar2020harmless,bartlett2020benign,hastie2022surprises,
chatterji2022foolish,shamir2023implicit,tsigler2023benign,zou2023sgd}.
Section~\ref{sec:related} organizes the closest statistical results by the spectral
regime they address.

The paper makes four contributions.
\begin{enumerate}[leftmargin=2em,itemsep=0.35em]
    \item We identify a sharp-floor regime of Gaussian random-design prediction
    in which rank selection alone is provably insufficient.  When the aggregate
    covariance tail creates a sharp floor comparable to the head scale,
    ordinary PCR has a prediction-risk lower bound uniform over all ranks,
    whereas the conditional risk of dPCR is asymptotically negligible relative
    to that of the best ordinary PCR rank.  Neither scale need be constant
    order.

    \item We explain this separation through an exact conditional prediction-risk
    decomposition.  The tail contribution that inflates empirical denominators
    is governed by first spectral mass, whereas the clean prediction cost of
    correction is governed by squared spectral mass.  Thus the floor can be
    nonnegligible while remaining cheap to remove.

    \item A same-sample trimmed mean of the empirical sub-cluster estimates the
    floor within its operator width and attains the oracle dPCR upper-bound rate
    at a prespecified rank.  The separation persists under approximate
    predictive alignment when the tail prediction-energy fraction vanishes.

    \item A separate pointwise fixed-aspect analysis gives the risk-optimal
    scalar correction in a one-spike model and a deterministic risk functional
    for fixed spectral cutoffs.  The optimal positive correction strictly
    improves ordinary rank-$1$ PCR, but mean-floor subtraction is generally not
    optimal for a broad Marchenko--Pastur bulk; it becomes exact in the
    collapsing-bulk, matched-scale limit.
\end{enumerate}

\paragraph*{Organization}
A detailed theorem roadmap begins in Section~\ref{sec:sharp}, where the formal results
start.  Section~\ref{sec:related} positions the contribution.  Section~\ref{sec:setup} defines PCR,
dPCR, and risk.  Section~\ref{sec:sharp} establishes the sharp random-floor theory and
its flat-tail consequences.
Section~\ref{sec:fixed} develops the proportional theory.  Section~\ref{sec:simulations}
reports simulations, and Section~\ref{sec:discussion} records limitations
and extensions.

\section{Related work and novelty boundaries}
\label{sec:related}

\subsection{Spectral regimes in existing theory}

\paragraph*{Regime A: the floor is absent or negligible}
Operator-regularization analyses of spectral cutoff control empirical
regularization through approximation to a population operator
\citep{dicker2017kernel}.  Likewise, nonasymptotic empirical-PCR bounds compare
sample and population spectral quantities under concentration or alignment
conditions \citep{huang2022dimensionality,hucker2023prediction}.  When the
retained scale dominates this approximation error, any additive tail shift is
lower order and the cutoff is the relevant regularization parameter.  Average
random-design calculations with known population directions similarly study
how risk changes with the number of retained variables, rather than a persistent
shift in their inverse weights \citep{xu2019variables}.
The latent-factor analysis of \citet{bing2021latent} likewise treats adaptive
rank selection in a strong-head regime where the covariate-noise scale is lower
order than the retained factor eigenvalues.

\paragraph*{Regime B: a nonnegligible floor acts as self-induced regularization}
When the effective-rank
condition of \citet{hucker2023prediction} fails, empirical eigenvalues acquire
an upward bias interpreted as self-induced regularization.  In related
head--tail analyses, a large tail trace
stabilizes variance in weak empirical modes, while excessive implicit
regularization can make the risk-optimal explicit ridge parameter negative
\citep{bartlett2020benign,tsigler2023benign}.  A
floor can therefore stabilize weak modes and simultaneously over-attenuate
predictive terms whose clean scale is comparable to it.

\paragraph*{Regime C: truncation denoises low-rank or corrupted covariates}
The error-in-variables and incomplete-observation literature gives guarantees
under noisy, missing, or semirandom covariates
\citep{agarwal2021robustness,bhaskara2021semirandom}; establishes minimum-norm identification
and out-of-sample prediction for fixed designs, with applications to synthetic
controls \citep{agarwal2025identification}; and develops online regularized PCR
for adaptive panel data \citep{agarwal2023adaptive}.  In these regimes,
truncation protects prediction from exogenous measurement error or temporal
adaptation.  Here the covariates are clean, and
many individually weak coordinates instead create an endogenous, nearly scalar
amount in sample space that remains inside every retained denominator.

\paragraph*{Regime D: the tail distortion is broad and non-scalar}
Proportional random-matrix theory does recognize persistent spectral distortion.
\citet{gedon2024pcr} study generalization and distribution shift in a spiked
covariance model, while \citet{green2025pcr} derive exact asymptotic prediction
risk for general population spectra using eigenvector-overlap measures.  At
fixed aspect ratio, however, the Marchenko--Pastur bulk has order-one relative
width.  A scalar mean-floor subtraction is then not exact; the
prediction-optimal inverse weighting depends on spectral location and overlap.

\subsection{The missing regime}

Let $h$ be the retained head rank and write the population eigenvalues as
$\lambda_1\ge\cdots\ge\lambda_p>0$.  The predictive head may be
heterogeneous; we assume only that $\lambda_1/\lambda_h$ remains bounded and
use its lower edge $\lambda_h$ as the reference head scale.  Anticipating the
notation of the main theorem, define
\begin{equation}
T_1=\sum_{j>h}\lambda_j,
\qquad
T_2=\sum_{j>h}\lambda_j^2,
\qquad
a=\frac{T_1}{n},
\qquad
\delta=\norm{K_T^{(1)}-aI_n}_{\op},
\label{eq:missing-regime-quantities}
\end{equation}
where $K_T^{(1)}=X_TX_T^\top/n$ is the tail sample-space Gram matrix.
For an independent Gaussian or sub-Gaussian tail, the leading width scale is
$\sqrt{T_2/n}$, with an additional largest-tail-eigenvalue term.  The regime of
interest is therefore summarized by
\begin{equation}
a\asymp\lambda_h,
\qquad
\frac{\delta}{a}\longrightarrow0,
\qquad
\frac{T_2}{n\lambda_h^2}\longrightarrow0,
\label{eq:missing-regime}
\end{equation}
For the Gaussian sufficient conditions below, the largest-coordinate part of
floor sharpness is
$n\lambda_{h+1}/T_1=\lambda_{h+1}/a\to0$.  Hence, when
$a\asymp\lambda_h$, the same condition entails
$\lambda_{h+1}/\lambda_h\to0$.  This is the head--tail boundary separation
implied by the sharp-floor regime; it is not an eigengap among the retained
head eigenvalues.  If $E_H$ is the prediction energy carried by the matched head and
$V_H$ is its ordinary estimation-variance scale, detectability requires
schematically
\begin{equation}
E_H\gg V_H+\frac{T_2}{n\lambda_h^2}E_H.
\label{eq:missing-regime-detectability}
\end{equation}
Thus the mechanism requires \emph{large trace, small squared mass, and matched
signal scale}: the trace creates an order-$\lambda_h$ denominator shift, while
the squared mass makes that shift sharp and its removal affordable.

This statement is scale-free.  With the effective tail dimension and matched
head-to-floor ratio
\begin{equation}
d_1=\frac{T_1^2}{T_2},
\qquad
\kappa=\frac{\lambda_h}{a},
\label{eq:effective-tail-dimension}
\end{equation}
the normalized leakage is exactly
\begin{equation}
\frac{T_2}{n\lambda_h^2}=\frac{n}{d_1\kappa^2}.
\label{eq:effective-tail-leakage}
\end{equation}
Consequently, the absolute sequences $a$ and $\lambda_h$ may both vanish,
both diverge, or vary irregularly.  What matters is their ratio, the effective
tail dimension, and the prediction energy relative to estimation noise.

For a flat tail with $m$ coordinates having the common population eigenvalue
$\lambda_{h+1}$,
\begin{equation}
\frac{m\lambda_{h+1}}{n}\asymp\lambda_h,
\qquad
\frac{m\lambda_{h+1}^2}{n}=o(\lambda_h^2).
\label{eq:flat-tail-missing-regime}
\end{equation}
The first relation makes the floor nonnegligible, and together the two imply
$n/m\to0$.  For a flat tail, ultra-high dimension is therefore structural
rather than cosmetic.  More generally, the driver is
$d_1/n\to\infty$, together with control of the largest tail
coordinate: many heterogeneous weak directions can create the same sharp
floor even when the literal tail dimension is not the informative quantity.
Table~\ref{tab:missing-regime} summarizes the resulting regime map.

\begin{table}[!tbp]
\centering
\small
\caption{The missing quadrant.  Floor size is measured relative to a predictive
head scale $\lambda_h$, and sharpness by $\delta/a$.  A third axis---the squared
tail mass $T_2/(n\lambda_h^2)$---determines whether correction is affordable.}
\label{tab:missing-regime}
\vspace{0.35em}
\begin{tabular}{>{\raggedright\arraybackslash}p{0.18\linewidth}
>{\raggedright\arraybackslash}p{0.34\linewidth}
>{\raggedright\arraybackslash}p{0.38\linewidth}}
\toprule
Tail floor & Sharp tail: $\delta/a\to0$ & Broad tail: $\delta/a\asymp1$ \\
\midrule
Negligible: $a/\lambda_h\to0$
& Operator and perturbative PCR theory: rank selection controls the leading
risk \citep{dicker2017kernel,huang2022dimensionality}.
& Distortion can be spectrally visible but remains lower order for matched-head
prediction; correction is unnecessary. \\
Matched: $a\asymp\lambda_h$
& \textbf{This paper's main regime:} a nonnegligible scalar floor creates
constant attenuation, while $T_2/(n\lambda_h^2)\to0$ makes removal affordable.
& Fixed-aspect random-matrix regime: the distortion is persistent but
non-scalar, so mean subtraction is only approximate
\citep{gedon2024pcr,green2025pcr}. \\
\bottomrule
\end{tabular}
\end{table}

\paragraph*{Adjacent corrections and the novelty boundary}
The closest direct antecedent is \citet{dicker2013oneshot}.  They study a
one-factor, fixed-sample high-dimensional model with an explicitly isotropic
residual and correct rank-one PCR using the gap between the leading and
trailing sample eigenvalues.  In that setting, the trailing eigenvalue
estimates the same flat sample-space floor targeted here, making their method
a rank-one flat-tail prototype of de-flooring.  The present theory develops
the growing-sample prediction problem with several predictive components and
heterogeneous clean covariance tails.  It proves common-floor formation,
separates that floor from clean prediction leakage, estimates the floor from
the same sample, permits approximate predictive alignment, and compares with
ordinary PCR uniformly over every rank.

Classical spectral cutoff or truncated SVD changes only the support of the
inverse filter: retained terms keep weight $1/\widehat\mu_i$, and discarded
terms receive weight zero \citep{engl1996regularization}.  Truncated total least
squares instead assumes errors in both the coefficient matrix and response and
truncates small singular values of their augmented matrix $[X,y]$
\citep{hansen1997ttls}; its correction is response-dependent and generally
changes both subspace and weights.  By contrast, dPCR keeps an $X$-only
eigenspace, is conditionally linear in $y$, and estimates its common shift from
the covariate tail.  Corrected Gram matrices for corrupted covariates
\citep{loh2011noisy} and optimal covariance eigenvalue shrinkage
\citep{donoho2018shrinkage} address still different observation or loss models.
Our contribution is a general prediction-risk theory for an endogenous floor
created by a sharp, high-effective-rank clean tail: an exact risk decomposition,
a PCR lower bound uniform over all ranks, a vanishing risk-ratio theorem, a
same-sample floor estimator, and a fixed-aspect boundary where scalar
subtraction ceases to be exact.  The spiked-covariance calculations use
classical outlier and eigenvector-overlap theory
\citep{baik2006spiked,paul2007spiked,ledoit2011eigenvectors}.

\section{Setup, estimators, and the two prediction-risk controls}
\label{sec:setup}

Let $X_0\in\R^{n\times p}$ and $\beta_0^\star\in\R^p$ denote the design
and coefficient in the original feature coordinates, and let
$\Sigma_0=\E(x_0x_0^\top)$.  Write the population eigendecomposition as
\begin{equation}
\Sigma_0
=V\diag(\lambda_1,\ldots,\lambda_p)V^\top,
\qquad
\lambda_1\ge\cdots\ge\lambda_p\ge0,
\label{eq:population-diagonalization}
\end{equation}
and rotate to population spectral coordinates:
\begin{equation}
X=X_0V,
\qquad
\beta^\star=V^\top\beta_0^\star,
\qquad
\Sigma=V^\top\Sigma_0V
=\diag(\lambda_1,\ldots,\lambda_p).
\label{eq:spectral-coordinates}
\end{equation}
Because $X_0\beta_0^\star=X\beta^\star$, the model becomes
\begin{equation}
    y=X\beta^\star+\varepsilon,
    \qquad \varepsilon\sim N(0,\sigma_\varepsilon^2I_n),
    \qquad X\in\R^{n\times p},\quad p\ge n.
    \label{eq:model}
\end{equation}
Along a triangular array, $\beta_n^\star$ is deterministic; the same arguments
apply if it is random but independent of the design.
PCR and dPCR predictions, as well as population prediction risk, are unchanged
by this orthogonal reparametrization.  We therefore use the diagonal
population coordinates in \eqref{eq:spectral-coordinates} throughout.  Write
\begin{equation}
\widehat\Sigma=\frac{X^\top X}{n}
=\sum_{i=1}^{q}\widehat\mu_i\widehat v_i\widehat v_i^\top,
\qquad
K=\frac{XX^\top}{n}
=\sum_{i=1}^{q}\widehat\mu_i\widehat u_i\widehat u_i^\top,
\label{eq:eigendecomp}
\end{equation}
where $q=\operatorname{rank}(X)$ and
$\widehat\mu_1\ge\cdots\ge\widehat\mu_q>0$.  The population prediction risk
and its response-noise average conditional on $X$ are
\begin{align}
\Risk_{\rm pop}(\widehat\beta)
&=(\widehat\beta-\beta^\star)^\top
\Sigma(\widehat\beta-\beta^\star),
\label{eq:pop-risk}\\
\Rbar_X(\widehat\beta)
&=\E_\varepsilon\{\Risk_{\rm pop}(\widehat\beta)\mid X\}.
\label{eq:conditional-risk}
\end{align}

\begin{definition}[Ordinary PCR]
For $0\le r\le q$,
\begin{equation}
\widehat\beta_{\PCR}(r)
=\sum_{i=1}^{r}
\frac{\widehat v_i^\top X^\top y/n}{\widehat\mu_i}\widehat v_i.
\label{eq:pcr}
\end{equation}
Its empirical spectral multiplier relative to the ridgeless fit is
$f_{\PCR,r}(\widehat\mu_i)=\1\{i\le r\}$.
\end{definition}

\begin{definition}[De-floored principal component regression]
For a correction level $a_0\ge0$ and rank $0\le r\le q$, define
\begin{equation}
\widehat\beta_{\dPCR}(a_0,r)
=\sum_{i=1}^{r}
\frac{\widehat v_i^\top X^\top y/n}{\widehat\mu_i-a_0}\widehat v_i.
\label{eq:dpcr}
\end{equation}
Given a margin $g_n>0$, the admissible pairs are
\begin{equation}
\mathcal A_{\dPCR}(X;g_n)
=\{(a_0,r):0\le a_0\le\widehat\mu_r-g_n\},
\label{eq:admissible}
\end{equation}
with the convention that $r=0$ is always admissible.  The empirical multiplier
is
\begin{equation}
f_{a_0,r}(\widehat\mu_i)
=\frac{\widehat\mu_i}{\widehat\mu_i-a_0}\1\{i\le r\}.
\label{eq:dpcr-filter}
\end{equation}
The margin only screens candidates whose corrected denominator is too close to
zero.  Once $(a_0,r)$ is admissible, $g_n$ does not enter the fitted estimator.
For validation over a finite candidate grid, a convenient relative screen is
$\widehat\mu_r-a_0\ge\eta\widehat\mu_r$ with a fixed small $\eta>0$; the
simulations use $\eta=0.1$.
\end{definition}

\begin{center}
\fcolorbox{black}{gray!5}{%
\begin{minipage}{0.93\linewidth}
\small
\textbf{Implementation of dPCR.}
Compute $(\widehat\mu_i,\widehat v_i)_{i=1}^n$ from $X^\top X/n$.  Given a
predictive rank $h<n$, estimate the floor by
\[
\widehat a_h=\frac{1}{n-h}\sum_{i=h+1}^n\widehat\mu_i.
\]
If $\widehat\mu_h-\widehat a_h\ge g_n>0$, fit
$\widehat\beta_{\dPCR}(\widehat a_h,h)$ using \eqref{eq:dpcr}.  When $h$ is
unknown, select among admissible $(a_0,r)$ candidates using an independent
validation sample.
\end{minipage}}
\end{center}

The two parameters solve different prediction-risk problems.  The rank $r$
decides which empirical terms enter the estimator, whereas $a_0$ changes their
inverse weights.
Ordinary PCR is the boundary case $a_0=0$.  Taking $r=q$ produces a full
shifted inverse and can be unstable; dPCR is useful because the cutoff can
enforce a corrected spectral gap.

\paragraph*{How the assumptions are organized}
The observation model \eqref{eq:model} and the risk definition
\eqref{eq:pop-risk} are used throughout.  All later assumption blocks are
local rather than cumulative.  The sharp-floor random-design theorem uses the
standard Gaussian head--tail model in
Assumptions~\ref{ass:random-tail} and \ref{ass:gaussian-head}; its realized-design verification is
given in Appendix~\ref{app:proofs}.  Assumption~\ref{ass:sharp-separation} collects the
triangular-array rate conditions under which the finite-sample comparison
yields separation.  The proportional theory has its own fixed-aspect block.

\subsection{Scalar-center mechanism: rank truncation does not correct denominator inflation}

Let $u\in\R^n$ have unit norm and consider a rank-one realized head
$K_H=s_1uu^\top$.  When the tail Gram satisfies
$K_T^{(1)}=aI_n+\Delta$ with $\norm{\Delta}_{\op}\ll s_1+a$, its scalar-center
approximation is
\begin{equation}
K\simeq s_1uu^\top+aI_n,
\qquad
y=\sqrt{ns_1}\,\beta^\star u+\varepsilon.
\label{eq:one-head}
\end{equation}
At this center, rank-$1$ ordinary PCR has structural head coefficient
\begin{equation}
\widehat\beta_{H,\PCR}
=\frac{s_1}{s_1+a}\beta^\star
+\frac{\sqrt{s_1}}{\sqrt n(s_1+a)}u^\top\varepsilon,
\label{eq:one-head-pcr}
\end{equation}
whereas dPCR with $a_0=a$ has
\begin{equation}
\widehat\beta_{H,\dPCR}
=\beta^\star+\frac{1}{\sqrt{ns_1}}u^\top\varepsilon.
\label{eq:one-head-dpcr}
\end{equation}
These displays expose the leading denominator effect only.  The random-design
theory below controls the perturbation $\Delta$, the empirical subspace, and
the clean prediction cost of the accompanying tail loadings.

\section{Sharp random floors}
\label{sec:sharp}

\paragraph*{Theory roadmap}
The theory proceeds by design regime rather than by proof technique.
Table~\ref{tab:theorem-roadmap} links each local assumption block to the results
that use it and to the prediction-risk conclusion to carry forward.  These
blocks are alternatives, not cumulative requirements: in particular, the
fixed-aspect theory does not inherit the sharp-tail conditions.

Three questions determine the regime.  Is the tail floor comparable to the
predictive head scale?  Is the tail Gram sufficiently close to a scalar
matrix?  Is the squared tail mass small enough that removing the floor is
cheap in population prediction risk?  Rank selection addresses omitted
signal; scalar de-flooring addresses attenuation only when the latter two
conditions make a common denominator correction meaningful.

\paragraph*{Spectral notation}
We use one notation system throughout the sharp-floor theory:
$\lambda_j$ always denotes the $j$th population covariance eigenvalue,
$s_j$ an eigenvalue of the realized head Gram $X_H^\top X_H/n$, and
$\widehat\mu_i$ an empirical eigenvalue of the full sample covariance
$X^\top X/n$.  The diagonal population blocks are
$\Lambda_H=\diag(\lambda_1,\ldots,\lambda_h)$ and
$\Lambda_T=\diag(\lambda_{h+1},\ldots,\lambda_p)$, while
$D_H=\diag(s_1,\ldots,s_h)$ records the realized head spectrum.  In the flat-tail
specialization, the common tail eigenvalue remains denoted $\lambda_{h+1}$;
no second tail sequence or renamed head endpoints are introduced.

\begin{table}[!tbp]
\centering
\footnotesize
\setlength{\tabcolsep}{3pt}
\caption{Roadmap to the main theoretical results.  The last column states the
prediction-risk message rather than the proof technique.}
\label{tab:theorem-roadmap}
\vspace{0.35em}
\begin{tabular}{>{\raggedright\arraybackslash}p{0.12\linewidth}
>{\raggedright\arraybackslash}p{0.13\linewidth}
>{\raggedright\arraybackslash}p{0.15\linewidth}
>{\raggedright\arraybackslash}p{0.46\linewidth}}
\toprule
Design regime & Active assumption blocks & Main results & Prediction-risk message \\
\midrule
Gaussian sharp floor
& Assms.~\ref{ass:random-tail}--\ref{ass:gaussian-head};
Assm.~\ref{ass:sharp-separation} for the asymptotic corollaries
& Thm.~\ref{cor:gaussian-head}; Cors.~\ref{cor:sharp-separation},
\ref{cor:plugin-floor}, and \ref{cor:approx-alignment}
& Under effective-dimension and largest-coordinate control, dPCR separates
from every PCR rank.  A trimmed-mean plug-in floor attains the oracle rate at a
prespecified rank, and the separation is robust to vanishing tail prediction
energy; see Figures~\ref{fig:gaussian-sharp} and \ref{fig:heterogeneous-ranks}.
Figure~\ref{fig:adaptive-validation} separately tests joint validation tuning. \\
Flat Gaussian tail
& Assms.~\ref{ass:random-tail}--\ref{ass:gaussian-head} with
$\lambda_{h+1}=\cdots=\lambda_p$
& Cor.~\ref{cor:flat-tail}
& When $m_n/n\to\infty$, the Wishart tail collapses in operator norm around
$a_nI_n$.  The scale-free and power-law rates follow directly from the sharp
random-floor theorem; see Figure~\ref{fig:gaussian-sharp}. \\
Fixed aspect
& Assm.~\ref{ass:fixed-aspect}
& Thms.~\ref{thm:fixed-rank1} and \ref{thm:fixed-functional}
& Scalar correction can improve PCR, but a broad Marchenko--Pastur bulk is not
its mean: the optimal correction also depends on spectral location and
eigenvector overlap; see Figures~\ref{fig:matched} and \ref{fig:fixed-aspect-simulations}. \\
\bottomrule
\end{tabular}
\end{table}

The sharp random-design proof uses a first-order cluster bound because its
squared risk cost already matches the baseline tail leakage; a more delicate
head-sandwiched expansion would not improve the total rate.  The fixed-aspect
section instead retains the full Marchenko--Pastur geometry and shows why a
broad bulk cannot be replaced by its mean.

\paragraph*{Head and signal alignment}
In the population spectral coordinates fixed in Section~\ref{sec:setup}, the head
consists of the first $h$ population eigendirections and the tail consists of
the rest.  Signal alignment means that most prediction energy
$\sum_{j=1}^p\lambda_j(\beta_j^\star)^2$ lies in the head.  For the cleanest
separation statement, we use the exact-alignment benchmark
$\beta_T^\star=0$.  Corollary~\ref{cor:approx-alignment} replaces this benchmark by a
quantitative prediction-energy condition: if
$E_T=\|\beta_T^\star\|_{\Lambda_T}^2=o(E_H)$, the risk-ratio bound gains only
$E_T/E_H$.  The condition is stated entirely in prediction norm.

The proof strategy is adapted directly from the head--tail architecture of
benign-overfitting theory
\citep{bartlett2020benign,tsigler2023benign}.  As in that literature, we first
diagonalize the population covariance, separate a low-dimensional predictive
head from a high-effective-rank tail, prove deterministic risk algebra on a
named design event, and only then use concentration to verify the event.  The
interpretation is reversed at one crucial point: benign-overfitting arguments
use the tail Gram matrix as helpful self-regularization, whereas here the same
approximately scalar contribution can over-regularize the predictive head.
The two tail matrices in \eqref{eq:two-tail-grams} then separate the trace that
creates this implicit regularization from the squared spectral mass that makes
its removal costly in clean prediction.

The cutoff makes the resulting negative-ridge calculation simpler than full
negative ridge.
Let $\widehat\Pi_r=\sum_{i\le r}\widehat v_i\widehat v_i^\top$ be the
feature-space PCR projector.  Whenever the corrected denominators are
positive, dPCR is equivalently
\begin{equation}
\widehat\beta_{\dPCR}(a,r)
=\argmin_{\beta\in\operatorname{range}(\widehat\Pi_r)}
\left\{
\frac{\norm{y-X\beta}^2}{2n}-\frac a2\norm\beta^2
\right\}.
\label{eq:projected-negative-ridge}
\end{equation}
Thus it is negative ridge restricted to an empirical PCR range.  The cutoff
removes the lower empirical directions that would otherwise approach the
negative-ridge pole.  The remaining proof combines the benign-overfitting
separation of deterministic algebra from concentration with an ordinary
cluster perturbation, an exact ridge-style conditional-risk decomposition, and
an all-rank lower bound specific to PCR.

\subsection{Heterogeneous Gaussian tail}

We first treat a nonflat tail; the flat model will be a corollary.  The
setup is the standard Gaussian random-design decomposition used in
head--tail analyses of benign overfitting.  In the population spectral
coordinates fixed in Section~\ref{sec:setup}, write
$X=Z\Sigma^{1/2}=[X_H\;X_T]$, with the first $h$ population
eigendirections in $X_H$ and the remaining directions in $X_T$.  The main
random-design result is Theorem~\ref{cor:gaussian-head}.  The realized-design
comparison and its Gaussian verification are deferred to
Appendix~\ref{app:proofs}, so the main text can state the prediction theorem
before introducing proof-specific operators.

\begin{assumption}[Gaussian tail block]
\label{ass:random-tail}
Let $1\le h<n$, let $m=p-h\ge n$, and
\begin{equation}
X_T=Z_T\Lambda_T^{1/2},
\qquad
\Lambda_T=\diag(\lambda_{h+1},\ldots,\lambda_p),
\label{eq:heterogeneous-design}
\end{equation}
where the columns $z_j$ of $Z_T$ are independent $N(0,I_n)$ vectors and
$\lambda_{h+1}\ge\cdots\ge\lambda_p>0$.  Thus $\Sigma_T=\Lambda_T$ in the
coordinates of Section~\ref{sec:setup}.
Conditional on the head, $Z_T$ has the stated law.  Write
$\beta^\star=(\beta_H^\star,\beta_T^\star)$ and assume throughout that the
coefficient sequence is deterministic, or independent of the design.
\end{assumption}

\begin{assumption}[Gaussian predictive head]
\label{ass:gaussian-head}
Let
\begin{equation}
\Lambda_H=\diag(\lambda_1,\ldots,\lambda_h),
\qquad \lambda_1\ge\cdots\ge\lambda_h>\lambda_{h+1},
\label{eq:population-head-spectrum}
\end{equation}
and let $X_H=Z_H\Lambda_H^{1/2}$, where $Z_H$
has independent $N(0,1)$ entries and is independent of $Z_T$.  Take
$\Sigma_H=\Lambda_H$, so $\Sigma=\diag(\Lambda_H,\Lambda_T)$.  The individual
head eigenvalues may be unequal; assume only
$\lambda_1/\lambda_h\le C_H$ for a fixed $C_H<\infty$.
Write $s_1\ge\cdots\ge s_h>0$ for the eigenvalues of the realized head Gram
$X_H^\top X_H/n$, and set $D_H=\diag(s_1,\ldots,s_h)$.
\end{assumption}

The independence of $Z_H$ and $Z_T$ is not an additional structural
restriction within the Gaussian model.  After rotation into the population
eigenbasis, the covariance is block diagonal, so the Gaussian head and tail
coordinates are uncorrelated and hence independent.  For non-Gaussian
designs, the same population diagonalization need not produce independent
blocks.

Suppress the triangular-array index and define
\begin{align}
T_1&=\sum_{j>h}\lambda_j,
&T_2&=\sum_{j>h}\lambda_j^2,
&T_4&=\sum_{j>h}\lambda_j^4,
\label{eq:tail-moments}\\
a&=\frac{T_1}{n},
&b_T&=\frac{T_2}{n},
&d_1&=\frac{T_1^2}{T_2},
&d_2&=\frac{T_2^2}{T_4}.
\label{eq:tail-effective-dimensions}
\end{align}
The two sample-space matrices needed for estimation and clean prediction are
\begin{equation}
K_T^{(1)}=\frac{X_TX_T^\top}{n}
=\frac1n\sum_{j>h}\lambda_jz_jz_j^\top,
\qquad
K_T^{(2)}=\frac{X_T\Lambda_TX_T^\top}{n}
=\frac1n\sum_{j>h}\lambda_j^2z_jz_j^\top.
\label{eq:two-tail-grams}
\end{equation}
Their expectations are $aI_n$ and $b_TI_n$, respectively.  The first matrix
creates the empirical denominator floor; the second is the tail readout in
population prediction risk.  This distinction is why both $d_1$ and $d_2$
appear.

For $t\ge0$, define
\begin{align}
q_1(t)
&=\sqrt{\frac{n+t}{d_1}}
+\frac{(n+t)\lambda_{h+1}}{T_1},
\label{eq:q1}\\
q_2(t)
&=\sqrt{\frac{n+t}{d_2}}
+\frac{(n+t)\lambda_{h+1}^2}{T_2}.
\label{eq:q2}
\end{align}

For the Gaussian model, define the head and tail prediction energies, the head
noise scale, and the floor-to-head ratio
\begin{equation}
\begin{aligned}
E_H&=\beta_H^{\star\top}\Lambda_H\beta_H^\star,
&E_T&=\beta_T^{\star\top}\Lambda_T\beta_T^\star,
&\kappa&=\frac{\lambda_h}{a},\\
V_H&=\frac{\sigma_\varepsilon^2h}{n}.
\end{aligned}
\label{eq:gaussian-head-scales}
\end{equation}
The head-concentration radius is
\begin{equation}
r_H(t)=\sqrt{\frac{h+t}{n}}+\frac{h+t}{n}.
\label{eq:head-concentration-radius}
\end{equation}
Small $r_H(t)$ ensures that the realized head Gram is comparable to its
population counterpart.

To distinguish the statistical model from the realized conditions used by the
risk comparison, let $S_H=X_H^\top X_H/n$.  The proofs use the derived design
event
\begin{equation}
\mathcal E(w,b)=
\left\{
\frac12\Lambda_H\preceq S_H\preceq\frac32\Lambda_H,
\quad
\norm{K_T^{(1)}-aI_n}_{\op}\le w,
\quad
\norm{K_T^{(2)}}_{\op}\le b
\right\}.
\label{eq:sharp-design-event}
\end{equation}
The deterministic comparison uses only this event together with a corrected
denominator margin.  Under the Gaussian model, head and weighted-tail
concentration verify it with probability at least $1-6e^{-t}$ for
$w=Ca q_1(t)$ and $b=b_T\{1+Cq_2(t)\}$ whenever $r_H(t)$ is sufficiently
small.  The event therefore provides the realized-design interface between
the deterministic risk comparison and its Gaussian concentration proof,
following the standard head--tail architecture of benign-overfitting theory.

\begin{assumption}[Asymptotic sharp matched separation]
\label{ass:sharp-separation}
Consider a triangular array satisfying
Assumptions~\ref{ass:random-tail} and \ref{ass:gaussian-head}, let $t_n=\log n$, and use the scales
in \eqref{eq:gaussian-head-scales}.  Assume:
\begin{enumerate}[label=(\roman*),leftmargin=2.2em]
\item $h_n=o(n)$ and $g_n\le c_g\lambda_{h_n,n}$ for a sufficiently small
fixed $c_g>0$;
\item $0<c_\kappa\le\kappa_n\le C_\kappa<\infty$;
\item $q_{1,n}(t_n)\to0$ and $q_{2,n}(t_n)=O(1)$;
\item $E_{H,n}>0$ eventually and
$\mathsf{SNR}_{H,n}:=E_{H,n}/V_{H,n}\to\infty$, with the convention
$\mathsf{SNR}_{H,n}=\infty$ when $V_{H,n}=0$.
\end{enumerate}
Corollary~\ref{cor:sharp-separation} first treats exact alignment,
$\beta_{T,n}^\star=0$, and Corollary~\ref{cor:approx-alignment} extends the result to
$E_{T,n}/E_{H,n}\to0$.
\end{assumption}

\paragraph*{Interpretation of Assumption \ref{ass:sharp-separation}}
Assumptions \ref{ass:random-tail} and \ref{ass:gaussian-head} specify the
finite-sample Gaussian design, while Assumption \ref{ass:sharp-separation}
selects the triangular-array regime in which the comparison becomes a
vanishing risk ratio.  Its four conditions have distinct roles.

Condition (i) keeps the predictive head low-dimensional relative to the
sample size, so $r_H(\log n)\to0$ and the realized head Gram concentrates
around the heterogeneous population head.  The head eigenvalues may differ;
the shell condition $\lambda_1/\lambda_h\le C_H$ makes the lower edge
$\lambda_h$ the common reference scale.  The bound
$g_n\le c_g\lambda_h$ is compatible with the fixed-order corrected
denominator margin supplied by the concentration event.

Condition (ii) matches the tail floor $a=T_1/n$ to that lower head edge.  This
is the regime in which the ordinary-PCR multiplier
$s_j/(s_j+a)$ remains attenuated by a nonvanishing amount across the head,
while subtracting $a$ targets the leading denominator distortion.  If
$\kappa=\lambda_h/a$ diverges, the floor is negligible at the head scale; if
$\kappa$ vanishes, relative floor-estimation errors are amplified at the
weakest retained head scale.  The matched regime isolates the setting in
which attenuation is persistent and scalar subtraction remains stable.

Condition (iii) separates floor formation from the clean prediction cost of
removing it.  The requirement $q_1(\log n)\to0$ yields
$K_T^{(1)}=aI_n+o_{\Pp}(a)$ in operator norm, so the same scalar correction
applies simultaneously to all retained head components, even for a
heterogeneous tail.  It also implies $n/d_1\to0$.  The bound
$q_2(\log n)=O(1)$ controls $K_T^{(2)}$, and hence
\[
\frac{n}{d_1}\{1+q_2(\log n)\}\longrightarrow0,
\]
which makes the tail loadings induced by de-flooring asymptotically cheap in
population prediction risk.  Finally, condition (iv) makes the shared head
estimation scale $V_H$ negligible relative to the head prediction energy
$E_H$.  Exact and approximate predictive alignment are handled separately by
$E_T=0$ and $E_T/E_H\to0$, respectively.

\paragraph*{Boundary and corrected-denominator margins}
The inverse calculation uses a corrected empirical cluster margin at the
head--tail boundary.  On the head-concentration event,
$s_h\asymp\lambda_h$, while
$\delta:=\norm{K_T^{(1)}-aI_n}_{\op}\le Ca q_1(t)$.  Hence the finite-sample
smallness condition gives
\begin{equation}
\widehat\mu_h-a\ge s_h-\delta\ge c\lambda_h,
\qquad
\widehat\mu_h-\widehat\mu_{h+1}
\ge s_h-2\delta\ge c\lambda_h,
\label{eq:corrected-cluster-margin}
\end{equation}
where the first inequality stabilizes the corrected denominators and the
second separates the retained cluster from the residual spectrum.  These are
the margins used by the inverse perturbation argument and by the
Weyl--Davis--Kahan cluster bound.

The Gaussian sufficient condition also has a population boundary-gap
consequence.  Directly from \eqref{eq:q1},
\begin{equation}
q_1(t)
\ge \left(1+\frac tn\right)\frac{\lambda_{h+1}}a
=\left(1+\frac tn\right)\kappa
\frac{\lambda_{h+1}}{\lambda_h}.
\label{eq:q1-implies-boundary-gap}
\end{equation}
Therefore, in the matched regime, $q_{1,n}(\log n)\to0$ implies
$\lambda_{h_n+1,n}/\lambda_{h_n,n}\to0$.  Since
$1+(\log n)/n\to1$, the resulting population boundary ratio is $o(1)$.

\begin{theorem}[Gaussian random-design sharp-floor risk and PCR barrier]
\label{cor:gaussian-head}
Under Assumptions~\ref{ass:random-tail} and \ref{ass:gaussian-head}, assume
$\beta_T^\star=0$.  If $r_H(t)\le c$,
$q_1(t)\le c\min(\kappa,1)$, and
$g_n\le c_g\lambda_h$ for sufficiently small constants $c,c_g$ depending
only on $C_H$, then with probability at least $1-6e^{-t}$,
\begin{align}
\Rbar_X\{\widehat\beta_{\dPCR}(a,h)\}
&\le C_{C_H}\left[
V_H+\frac{E_H+V_H}{\kappa^2}
\left\{q_1(t)^2+\frac{n}{d_1}\{1+q_2(t)\}\right\}
\right],
\label{eq:gaussian-head-dpcr-bound}\\
\inf_{0\le r\le n}\Rbar_X\{\widehat\beta_{\PCR}(r)\}
&\ge c_{C_H}\frac{E_H}{(1+C_H\kappa)^2}.
\label{eq:gaussian-head-pcr-bound}
\end{align}
Thus the theorem permits the heterogeneous population head
$\lambda_1\ge\cdots\ge\lambda_h$; its lower edge $\lambda_h$ is the reference
scale for the matched-floor and denominator-margin conditions.
\end{theorem}

This theorem separates the two effects of the tail.  The relative width
$q_1$ controls whether one scalar can correct all retained denominators,
whereas $n\{1+q_2\}/d_1$ controls the population prediction cost of the tail
loadings carried by those empirical components.

\begin{corollary}[Scale-free random-design separation]
\label{cor:sharp-separation}
Consider a triangular array satisfying
Assumption~\ref{ass:sharp-separation} and suppose $\beta_{T,n}^\star=0$.  Then
\begin{align}
\frac{\Rbar_X\{\widehat\beta_{\dPCR}(a_n,h_n)\}}
{\inf_r\Rbar_X\{\widehat\beta_{\PCR}(r)\}}
={}&O_{\Pp}\!\left[
\mathsf{SNR}_{H,n}^{-1}
+\{1+\mathsf{SNR}_{H,n}^{-1}\}
\right.\nonumber\\[-0.25em]
&\left.\qquad\times\left\{
q_{1,n}(t_n)^2+\frac{n}{d_{1,n}}\{1+q_{2,n}(t_n)\}
\right\}
\right]
\longrightarrow0.
\label{eq:sharp-rates}
\end{align}
The conclusion permits the head and floor scales to vanish, diverge, or vary
irregularly.  What matters is their ratio, the two effective dimensions, the
largest tail coordinates, and head prediction SNR.
\end{corollary}

Two data-relevant extensions preserve this headline conclusion.  First, the
same-sample sub-cluster mean $\widehat a_h$ estimates $a$ within the same width
$w$ and attains the oracle-floor rate; see Corollary~\ref{cor:plugin-floor}.  Second,
when $\rho_T=E_T/E_H\to0$, both the oracle and plug-in risk ratios gain only an
$O_{\Pp}(\rho_T)$ term; see Corollary~\ref{cor:approx-alignment}.

\paragraph*{Proof architecture}
The realized-head benchmark, projected-inverse identity, weighted-tail
concentration lemma, and deterministic all-rank PCR barrier are collected in
Appendix~\ref{app:proofs}.  Together they yield the floor-estimation,
approximate-alignment, and flat-tail results below.
\begin{remark}[Estimated floor]
\label{rem:random-estimated-floor}
Let $\widehat R(\widetilde a)$ use the same top-$h$ empirical projector but
subtract $\widetilde a$.  On the head-concentration and sharp-floor event of
Theorem~\ref{cor:gaussian-head}, if $|\widetilde a-a|\le c\lambda_h$, then
\begin{equation}
\norm{\widehat R(\widetilde a)-\widehat R(a)}_{\op}
\le C\frac{|\widetilde a-a|}{\lambda_h^2}.
\label{eq:random-floor-perturbation}
\end{equation}
Accordingly, \eqref{eq:gaussian-head-dpcr-bound} gains at most
$C_{C_H}(E_H+V_H)\{(\widetilde a-a)/\lambda_h\}^2$.
The trimmed-mean estimate in Corollary~\ref{cor:plugin-floor} satisfies this
perturbation condition on the same concentration event.
\end{remark}

\begin{corollary}[Same-sample plug-in floor]
\label{cor:plugin-floor}
Assume $h<n$ and define the trimmed-mean floor estimate
\begin{equation}
\widehat a_h=\frac{1}{n-h}\sum_{i=h+1}^{n}\widehat\mu_i.
\label{eq:plugin-floor}
\end{equation}
Under the hypotheses of Theorem~\ref{cor:gaussian-head}, on the same event of
probability at least $1-6e^{-t}$: every sub-cluster sample eigenvalue
satisfies $|\widehat\mu_i-a|\le\delta$ for $i>h$, where
$\delta=\norm{K_T^{(1)}-aI_n}_{\op}$; consequently
$|\widehat a_h-a|\le\delta\le Caq_1(t)$; the pair $(\widehat a_h,h)$ is
admissible; and
\begin{equation}
\Rbar_X\{\widehat\beta_{\dPCR}(\widehat a_h,h)\}
\le C_{C_H}\left[
V_H+\frac{E_H+V_H}{\kappa^2}
\left\{
q_1(t)^2+\frac{n}{d_1}\{1+q_2(t)\}
\right\}
\right].
\label{eq:plugin-floor-bound}
\end{equation}
Moreover, $(\widehat a_{h_n},h_n)$ may replace $(a_n,h_n)$ in
Corollary~\ref{cor:sharp-separation}.  The estimate may be replaced by any statistic
valued in the convex hull of the sub-cluster eigenvalues, such as their
median.  Given the retained rank, correct de-flooring therefore requires no
oracle knowledge of the tail.
\end{corollary}

\begin{corollary}[Approximate predictive alignment]
\label{cor:approx-alignment}
Consider a triangular array satisfying Assumption~\ref{ass:sharp-separation}, and allow
$\beta_n^\star=(\beta_{H,n}^\star,\beta_{T,n}^\star)$.  Define
\begin{equation}
E_{T,n}:=\beta_{T,n}^{\star\top}\Lambda_{T,n}\beta_{T,n}^\star,
\qquad
\rho_{T,n}:=\frac{E_{T,n}}{E_{H,n}},
\qquad
\eta_n:=q_{1,n}(t_n)^2+\frac{n}{d_{1,n}}
\{1+q_{2,n}(t_n)\}.
\label{eq:tail-prediction-leakage}
\end{equation}
If $\rho_{T,n}\to0$, then, for either the oracle floor
$\widetilde a_n=a_n$ or the same-sample plug-in floor
$\widetilde a_n=\widehat a_{h_n}$,
\begin{equation}
\frac{\Rbar_X\{\widehat\beta_{\dPCR}(\widetilde a_n,h_n)\}}
{\inf_r\Rbar_X\{\widehat\beta_{\PCR}(r)\}}
=O_{\Pp}\!\left[
\mathsf{SNR}_{H,n}^{-1}
+\{1+\mathsf{SNR}_{H,n}^{-1}\}\eta_n
+\rho_{T,n}
\right]
\longrightarrow0.
\label{eq:approx-alignment-ratio}
\end{equation}
Thus approximate alignment is measured only by the fraction of population
prediction energy in the tail.
\end{corollary}

\begin{remark}[Data-driven rank in the matched regime]
\label{rem:rank-selection}
The remaining oracle input is the rank.  Suppose additionally $h\le n/4$ and
that a known constant $\underline\kappa>0$ lower-bounds the realized clean
head-to-floor ratio $s_h/a$ on the event under consideration.  Let
$\widehat a_0$ be the median of all $n$ sample eigenvalues and
$\widehat h=\#\{i:\widehat\mu_i>(1+\underline\kappa/2)\widehat a_0\}$.  If
$q_1(t)\le c$ for a constant $c$ depending only on $\underline\kappa$, then
$\widehat h=h$ on the concentration event, and Corollary~\ref{cor:plugin-floor}
applies with $(\widehat a_{\widehat h},\widehat h)$.  If
    $\kappa=\lambda_h/a\ge c_\kappa$ and the head
concentration event is small enough to give
$s_h\ge\lambda_h/2$, one may take
$\underline\kappa=c_\kappa/2$.  Thus the threshold accounts explicitly for
the lower edge of the population head shell.
In the matched sharp regime the full pair is therefore computable from the
same sample, given a lower bound on the matched head-to-floor ratio.
\end{remark}

\begin{remark}[Validation outside the sharp regime]
\label{rem:validation}
When floor sharpness or a matched-ratio lower bound is in doubt---for
instance under the broad bulks of Section~\ref{sec:fixed}---the pair $(a_0,r)$ can
instead be selected on an independent validation sample over the admissible
grid \eqref{eq:admissible}.  The fixed-aspect functionals of Section~\ref{sec:fixed}
describe the limits that such selection targets.  We use the plug-in floor in
the sharp regime because it provably attains the oracle rate there.  A
validation oracle inequality for jointly selecting $(a_0,r)$ outside that
regime remains open.
\end{remark}

\subsection{Flat Gaussian tail}

\begin{corollary}[Flat Gaussian tail]
\label{cor:flat-tail}
Consider a triangular array satisfying
Assumptions~\ref{ass:random-tail} and \ref{ass:gaussian-head}, with
$\lambda_{h_n+1,n}=\cdots=\lambda_{p_n,n}$ for all $m_n$
tail coordinates.  Use the quantities in
\eqref{eq:gaussian-head-scales}, assume $E_{H,n}>0$, and set
$\gamma_n=m_n/n$ and
$\mathsf{SNR}_{H,n}=E_{H,n}/V_{H,n}$ when $V_{H,n}>0$, with the convention
$\mathsf{SNR}_{H,n}=\infty$ when $V_{H,n}=0$.  Suppose
$\beta_{T,n}^\star=0$ and
\begin{equation}
\gamma_n\to\infty,
\qquad
0<c_\kappa\le\kappa_n\le C_\kappa<\infty,
\qquad
h_n=o(n),
\qquad
g_n\le c_g\lambda_{h_n,n}.
\label{eq:flat-gaussian-conditions}
\end{equation}
Then
\begin{equation}
T_{1,n}=m_n\lambda_{h_n+1,n},
\qquad
T_{2,n}=m_n\lambda_{h_n+1,n}^2,
\qquad
d_{1,n}=d_{2,n}=m_n,
\qquad
a_n=\frac{m_n\lambda_{h_n+1,n}}{n}.
\label{eq:flat-tail-specialization}
\end{equation}
Moreover,
\begin{equation}
\frac{\norm{K_T^{(1)}-a_nI_n}_{\op}}{a_n}
=\norm{\frac{Z_TZ_T^\top}{m_n}-I_n}_{\op}
=O_{\Pp}\!\left(\sqrt{\frac{n}{m_n}}+\frac{n}{m_n}\right).
\label{eq:flat-wishart-collapse}
\end{equation}
Thus $a_nI_n$ is the operator-norm center of the tail Gram, not an exact
finite-sample identity.  Equivalently, the Wishart--Marchenko--Pastur bulk
collapses around its mean when $m_n/n\to\infty$.  The two concentration scales are
$q_{1,n}(t)=q_{2,n}(t)=
O\{\sqrt{(n+t)/m_n}+(n+t)/m_n\}$.  Consequently,
\begin{align}
\Rbar_X\{\widehat\beta_{\dPCR}(a_n,h_n)\}
&=O_{\Pp}\!\left[
V_{H,n}+\frac{E_{H,n}+V_{H,n}}
{\gamma_n\kappa_n^2}
\right],
\label{eq:flat-random-risk}\\
\frac{\Rbar_X\{\widehat\beta_{\dPCR}(a_n,h_n)\}}
{\inf_r\Rbar_X\{\widehat\beta_{\PCR}(r)\}}
&=O_{\Pp}\!\left[
\mathsf{SNR}_{H,n}^{-1}
+\frac{1+\mathsf{SNR}_{H,n}^{-1}}
{\gamma_n\kappa_n^2}
\right].
\label{eq:flat-random-ratio}
\end{align}
By Corollary~\ref{cor:plugin-floor} both displays
hold with the trimmed-mean plug-in floor $\widehat a_{h_n}$ in place of
$a_n$.  If $\mathsf{SNR}_{H,n}\to\infty$, then the ratio in
\eqref{eq:flat-random-ratio} tends to zero.
\end{corollary}

\paragraph*{Arbitrary covariance scales}
Choose any positive floor sequence $a_n$ and any $\gamma_n\to\infty$, set
$\lambda_{h_n+1,n}=a_n/\gamma_n$, and take
$\lambda_{h_n,n}=\kappa_na_n$ with $\kappa_n$ bounded above and below while
$\lambda_{1,n}/\lambda_{h_n,n}$ remains bounded.  If $h_n=o(n)$,
$E_{H,n}\gg\sigma_{\varepsilon,n}^2h_n/n$, and
$g_n\le c_g\lambda_{h_n,n}$, then Corollary~\ref{cor:flat-tail} applies and
\eqref{eq:flat-random-ratio} tends to
zero.  Hence $a_n$ and all head eigenvalues may vanish, diverge, or vary
irregularly.  For example, $\gamma_n=(\log n)^4$ gives a logarithmic
$O_{\Pp}\{(\log n)^{-4}\}$ tail contribution under matched scales.

\paragraph*{Power-law illustration}
For the illustrative choice
$\lambda_{h_n+1,n}=n^{-\theta}$, $m_n\asymp n^{1+\theta}$,
$a_n\asymp\lambda_{h_n,n}\asymp E_{H,n}\asymp1$ with a bounded heterogeneous
head shell, and
$\sigma_{\varepsilon,n}^2=O(1)$ with fixed $h_n$, the rate reduces to
\begin{equation}
\Rbar_X\{\widehat\beta_{\dPCR}(a_n,h_n)\}
=O_{\Pp}(n^{-\theta}+n^{-1}),
\qquad
\frac{\Rbar_X\{\widehat\beta_{\dPCR}(a_n,h_n)\}}
{\inf_r\Rbar_X\{\widehat\beta_{\PCR}(r)\}}
=O_{\Pp}(n^{-\theta}+n^{-1}).
\label{eq:power-law-illustration}
\end{equation}
The tail term dominates for $0<\theta<1$, while head variance dominates for
$\theta>1$.

\section{A fixed-aspect boundary: why the bulk mean is not enough}
\label{sec:fixed}

The sharp-floor approximation fails when the tail aspect stays fixed.  We use
the following local assumption block.  Only in this section, $\lambda$ denotes
the single population spike and $\tau$ the common flat-bulk eigenvalue; these
are scalar specializations of the population spectrum $\{\lambda_j\}$ used
above.

\begin{assumption}[Fixed-aspect one-spike model]
\label{ass:fixed-aspect}
There is one signal coordinate and $m$ flat tail coordinates, with
\begin{equation}
X=[\sqrt\lambda z\;\;\sqrt\tau Z_T],
\qquad \frac mn\to\gamma>1,
\qquad \beta^\star=(b,0,\ldots,0),
\qquad \Sigma=\diag(\lambda,\tau I_m).
\label{eq:fixed-model}
\end{equation}
Here $z$ and the entries of $Z_T$ are independent standard Gaussian variables,
$\lambda>\tau>0$, and $b\ne0$.
\end{assumption}
The tail Gram has Marchenko--Pastur law $F_{\gamma,\tau}$, normalized as the
probability law of its $n$ sample-space eigenvalues, with support
\begin{equation}
\ell_\gamma=\tau(\sqrt\gamma-1)^2,
\qquad
u_\gamma=\tau(\sqrt\gamma+1)^2,
\qquad
a=\gamma\tau.
\label{eq:mp-support}
\end{equation}
Its relative width does not vanish, so $K_T^{(1)}$ cannot be replaced by $aI_n$
\citep{marchenko1967distribution}.

\subsection{Separated one-spike formula}

Under Assumption~\ref{ass:fixed-aspect}, when
\begin{equation}
\lambda>\tau(1+\sqrt\gamma),
\label{eq:bbp}
\end{equation}
the leading sample eigenvalue and squared sample--population eigenvector overlap
converge to \citep{baik2006spiked,green2025pcr}
\begin{align}
\rho_\gamma(\lambda,\tau)
&=\lambda\left(1+\frac{\gamma\tau}{\lambda-\tau}\right),
\label{eq:outlier}\\
\chi_\gamma(\lambda,\tau)
&=\frac{1-\gamma\tau^2/(\lambda-\tau)^2}
{1+\gamma\tau/(\lambda-\tau)}.
\label{eq:overlap}
\end{align}

\begin{theorem}[Rank-$1$ proportional risk and optimal scalar correction]
\label{thm:fixed-rank1}
Under Assumption~\ref{ass:fixed-aspect} and \eqref{eq:bbp}, keep
$\lambda,\tau,b,$ and $\sigma_\varepsilon^2$ fixed as $n\to\infty$, and write
$\rho=\rho_\gamma(\lambda,\tau)$ and
$\chi=\chi_\gamma(\lambda,\tau)$.  For fixed $0\le a_0<\rho$, choose any fixed
$0<g<\rho-a_0$ and let
$c(a_0)=\rho/(\rho-a_0)$.  Then
\begin{align}
\Rbar_X\{\widehat\beta_{\PCR}(1)\}
&\xrightarrow{\Pp}
b^2\{\lambda(1-\chi)^2+\tau\chi(1-\chi)\},
\label{eq:fixed-pcr-rank1}\\
\Rbar_X\{\widehat\beta_{\dPCR}(a_0,1)\}
&\xrightarrow{\Pp}
b^2\{\lambda(c(a_0)\chi-1)^2
+\tau c(a_0)^2\chi(1-\chi)\}.
\label{eq:fixed-dpcr-rank1}
\end{align}
The $O_{\Pp}(n^{-1})$ response-noise term vanishes in these fixed-rank limits, and
$(a_0,1)$ is admissible with probability tending to one.  The
risk-minimizing multiplier and correction are
\begin{align}
c^\star
&=\frac{\lambda}{\lambda\chi+\tau(1-\chi)}
=\frac{\rho}{\lambda},
\label{eq:cstar}\\
a_0^\star
&=\rho(1-\chi)\left(1-\frac{\tau}{\lambda}\right)
=\rho-\lambda
=\frac{\gamma\tau\lambda}{\lambda-\tau},
\label{eq:astar}\\
\Rbar_X\{\widehat\beta_{\dPCR}(a_0^\star,1)\}
&\xrightarrow{\Pp}
b^2\frac{\lambda\tau(1-\chi)}
{\lambda\chi+\tau(1-\chi)}
=b^2\frac{\gamma\lambda\tau^2}{(\lambda-\tau)^2}.
\label{eq:fixed-opt-risk}
\end{align}
For every separated spike with $\lambda>\tau$, $c^\star>1$.  If $b\ne0$,
the optimally corrected rank-$1$ dPCR therefore strictly improves the limiting
signal risk of rank-$1$ PCR and the rank-$0$ estimator.  At fixed $\gamma$
the improvement is bounded by the prediction cost of tail content in the
leading empirical component.  The pointwise large-$\gamma$ limit of these
closed forms is given below.
\end{theorem}

The rank-$1$ boundary calculation has three prediction consequences.  The
risk-optimal positive scalar correction strictly improves ordinary
rank-$1$ PCR; the optimal correction generally differs from the bulk mean
because rescaling the head also magnifies persistent tail content in the
empirical eigenvector; and the optimal limiting risk in
\eqref{eq:fixed-opt-risk} remains positive at
fixed aspect.  The collapsing-bulk limit below then recovers mean-floor
subtraction, connecting this boundary calculation to the sharp-floor theorem.

\subsection{Scale normalization}

The fixed-aspect formulas also do not require an absolute constant-order spike.
Put
\begin{equation}
\xi=\frac{\lambda}{\tau},
\qquad
\rho^\circ_\gamma(\xi)
=\frac{\rho_\gamma(\lambda,\tau)}{\tau}
=\xi\left(1+\frac{\gamma}{\xi-1}\right),
\qquad
\alpha_0=\frac{a_0}{\tau}.
\label{eq:fixed-normalization}
\end{equation}
Then $\chi_\gamma$ depends only on $(\xi,\gamma)$,
$c(a_0)=\rho_\gamma^\circ/(\rho_\gamma^\circ-\alpha_0)$, and the rank-$1$
signal risk is
\begin{equation}
b^2\tau\left[
\xi\{c(a_0)\chi_\gamma-1\}^2
+c(a_0)^2\chi_\gamma(1-\chi_\gamma)
\right].
\label{eq:fixed-normalized-risk}
\end{equation}
Thus, for a triangular array with parameter tuple
$(m_n,\lambda_n,\tau_n,b_n,\sigma_{\varepsilon,n}^2)$ and $b_n\ne0$
eventually, define
$\gamma_n=m_n/n$, $a_n=\gamma_n\tau_n$,
$\xi_n=\lambda_n/\tau_n$, and
$\alpha_{0,n}=a_{0,n}/\tau_n$.  Suppose
$\gamma_n\to\gamma\in(1,\infty)$,
$\xi_n\to\xi>1+\sqrt\gamma$, and
$\alpha_{0,n}\to\alpha_0<\rho_\gamma^\circ(\xi)$.
After division by $b_n^2\tau_n$, the signal formulas depend only on
$(\xi_n,\gamma_n,\alpha_{0,n})$; the absolute scale $\tau_n$ may vanish or
diverge.  The strict limiting gap permits a scaled admissibility margin
$g_n=\bar g\tau_n$ for any fixed
$0<\bar g<\rho_\gamma^\circ(\xi)-\alpha_0$, after reducing $\bar g$ if
necessary.  Under this margin the response-noise contribution remains
$O_{\Pp}(\sigma_{\varepsilon,n}^2/n)$, and it is negligible relative to the
signal scale when
$\sigma_{\varepsilon,n}^2/n=o(b_n^2\tau_n)$; otherwise it must be retained.
The normalized optimal correction is
\begin{equation}
\frac{a_{0,n}^\star}{\tau_n}
=\rho_{\gamma_n}^\circ(\xi_n)
\{1-\chi_{\gamma_n}(\xi_n)\}
\left(1-\frac1{\xi_n}\right).
\label{eq:fixed-normalized-correction}
\end{equation}

As an algebraic sharp-floor limit, suppose $\gamma_n\to\infty$ and
$\kappa_n=\lambda_n/(\gamma_n\tau_n)=\xi_n/\gamma_n\to
\kappa\in(0,\infty)$.  The fixed-aspect closed forms give
\begin{equation}
\frac{\rho_{\gamma_n}}{a_n}\to1+\kappa,
\qquad
\chi_{\gamma_n}\to\frac{\kappa}{1+\kappa},
\qquad
\frac{a_{0,n}^\star}{a_n}\to1.
\label{eq:sharp-limit-optimal-floor}
\end{equation}
This pointwise large-$\gamma$ limit explains why the mean floor emerges as the
bulk becomes relatively sharp.  The nonasymptotic $p_n/n\to\infty$ transfer is
given separately by Theorem~\ref{cor:gaussian-head}.

Mean-floor subtraction $a_0=\gamma\tau$ need not equal $a_0^\star$; indeed,
$a_0^\star/a_0=\lambda/(\lambda-\tau)>1$ at fixed parameters.  Exact
cancellation of head attenuation would use multiplier $1/\chi$ and denominator
$\rho\chi$, while prediction-risk optimization uses denominator
\begin{equation}
d_{\rm pred}^\star
=\frac{\rho}{c^\star}
=\rho\frac{\lambda\chi+\tau(1-\chi)}{\lambda}
=\lambda.
\label{eq:mp-denominator}
\end{equation}
Thus the optimal corrected denominator in this flat one-spike model is exactly
the population spike.  The difference from exact head cancellation is the
price of tail loadings in clean prediction risk.

\subsection{General limiting cutoffs}

Complete the positive-spectrum eigenvectors
$\widehat v_1,\ldots,\widehat v_n$ to an orthonormal basis
$\widehat v_1,\ldots,\widehat v_p$ of $\R^p$, and set
$\widehat\mu_{n+1}=\cdots=\widehat\mu_p=0$ on the resulting null-space basis.
Define the full and positive
spike spectral measures
\begin{equation}
\nu_{\lambda,n}^{\mathrm{full}}
=\sum_{i=1}^{p}|\langle\widehat v_i,e_1\rangle|^2
\delta_{\widehat\mu_i},
\qquad
\nu_{\lambda,n}
=\sum_{i=1}^{n}|\langle\widehat v_i,e_1\rangle|^2
\delta_{\widehat\mu_i}
\ \Longrightarrow\ \nu_\lambda .
\label{eq:spike-measure}
\end{equation}

\begin{proposition}[Specialized signal spectral-measure limit]
\label{prop:green-romanov-import}
Under Assumption~\ref{ass:fixed-aspect}, the signal-weighted spectral-measure results of
\citet{green2025pcr} give, for every bounded continuous $\varphi$,
\begin{equation}
e_1^\top\varphi(\widehat\Sigma)e_1
\xrightarrow{\Pp}
\int\varphi(x)\,d\nu_\lambda^{\mathrm{full}}(x).
\label{eq:green-romanov-import}
\end{equation}
Here $\nu_\lambda^{\mathrm{full}}$ is a probability measure, its positive
restriction is $\nu_\lambda$, and its null-space mass is
\begin{equation}
\nu_\lambda^{\mathrm{full}}(\{0\})
=\left\{1+\frac{\lambda}{\tau(\gamma-1)}\right\}^{-1}>0.
\label{eq:fixed-null-mass}
\end{equation}
More precisely, the empirical signal measure $\widehat B_n$ in Section 5.1 of
that paper is exactly $\nu_{\lambda,n}^{\mathrm{full}}$ here.  Its Lemma 1 and
the weak-convergence argument immediately following it yield
\eqref{eq:green-romanov-import}, while its Lemma 3 gives the support and
outlier-atom structure.  Appendix~\ref{app:proofs} derives
\eqref{eq:fixed-null-mass} directly by a Sherman--Morrison calculation.
Above the transition,
$\nu_\lambda$ has an atom of mass $\chi$ at $\rho$; the remaining nonzero mass
is supported on the MP bulk.
\end{proposition}

The import's hypotheses are transparent in this specialization.  Gaussian
entries supply all required moments; the population spectral distribution
$(\delta_\lambda+m\delta_\tau)/(m+1)$ converges to the compactly supported
profile $\delta_\tau$ bounded away from zero; $\lambda$ is a single fixed
spike; and the deterministic signal direction $e_1$ has a convergent
signal-weighted measure.  Finally, $\gamma>1$ gives the regular positive MP
edges in \eqref{eq:mp-support}, separated from zero, while the cutoff below is
kept away from the edges and from atoms.  Thus restriction to $(0,\infty)$
gives \eqref{eq:spike-measure}.  This abstract definition avoids imposing a
Stieltjes-transform sign convention on the formulas below.

For a limiting cutoff $s$ that is not a mass point, let
$r_s=\#\{i:\widehat\mu_i\ge s\}$.  For admissible
$0\le a_0\le s-g$, define
\begin{equation}
f_{a_0,s}(x)=
\begin{cases}
x/(x-a_0),&x\ge s,\\
0,&x<s,
\end{cases}
\quad
C(a_0,s)=\int f_{a_0,s}\,d\nu_\lambda,
\quad
B(a_0,s)=\int f_{a_0,s}^2\,d\nu_\lambda.
\label{eq:fixed-functionals}
\end{equation}

\begin{theorem}[Pointwise one-spike fixed-aspect risk functional]
\label{thm:fixed-functional}
Under Assumption~\ref{ass:fixed-aspect}, for a fixed cutoff $s$ at positive distance
from the bulk edges, satisfying
$\nu_\lambda(\{s\})=0$, a fixed corrected margin $g>0$, and each fixed
$a_0\in[0,s-g]$,
\begin{align}
\Rbar_X\{\widehat\beta_{\dPCR}(a_0,r_s)\}
\xrightarrow{\Pp}
R_{\dPCR}^{(\gamma)}(a_0,s)
=&\ b^2\left[
\lambda\{C(a_0,s)-1\}^2
+\tau\{B(a_0,s)-C(a_0,s)^2\}
\right]
\nonumber\\
&+\sigma_\varepsilon^2\tau
\int_s^\infty\frac{x}{(x-a_0)^2}\,dF_{\gamma,\tau}(x).
\label{eq:fixed-functional-risk}
\end{align}
For ordinary PCR, $a_0=0$, $A(s)=\nu_\lambda([s,\infty))$, and
\begin{align}
\Rbar_X\{\widehat\beta_{\PCR}(r_s)\}
&\xrightarrow{\Pp}
R_{\PCR}^{(\gamma)}(s)
\nonumber\\
&=b^2\{\lambda(1-A(s))^2+\tau(A(s)-A(s)^2)\}
+\sigma_\varepsilon^2\tau\int_s^\infty x^{-1}\,dF_{\gamma,\tau}(x).
\label{eq:fixed-pcr-functional}
\end{align}
\end{theorem}

Let the extended limiting cutoff domain be
$\mathcal S_\gamma=[0,\infty]$, where $s=0$ retains every positive empirical
eigenvalue and $s=\infty$ is rank zero.  For any fixed parameters satisfying
Assumption~\ref{ass:fixed-aspect}, define
\begin{equation}
L_{\PCR}:=\inf_{s\in\mathcal S_\gamma}R_{\PCR}^{(\gamma)}(s).
\label{eq:fixed-pcr-positive-level}
\end{equation}
By \eqref{eq:fixed-null-mass},
$A(s)\le1-\nu_\lambda^{\mathrm{full}}(\{0\})<1$ uniformly over this domain.
The signal term in \eqref{eq:fixed-pcr-functional} is therefore at least
$b^2\lambda\{\nu_\lambda^{\mathrm{full}}(\{0\})\}^2$, and hence
$L_{\PCR}>0$.  Upgrading the pointwise result to
the finite-sample all-rank statement
\begin{equation}
\liminf_{n\to\infty}\inf_r\Rbar_X\{\widehat\beta_{\PCR}(r)\}
\ge L_{\PCR}
\label{eq:fixed-lower-target}
\end{equation}
requires uniform convergence over cutoffs and one-sided control at the bulk
edges.  The fixed-aspect theorem therefore yields pointwise cutoff limits only.

The fixed-aspect expression also quantifies the cost of PCR's fixed unit
multiplier.  For a
PCR projector with spike mass $A\in[0,1)$, set
$D(A)=A\{\tau+(\lambda-\tau)A\}$.  Its signal risk decomposes as
\begin{align}
R_{\rm opt\text{-}scale}(A)
&=b^2\frac{\lambda\tau(1-A)}{\tau+(\lambda-\tau)A},
\label{eq:optimal-scale-risk}\\
R_{\rm unit\text{-}excess}(A)
&=b^2\frac{A(\lambda-\tau)^2(1-A)^2}
{\tau+(\lambda-\tau)A},
\label{eq:unit-multiplier-excess}\\
R_{\rm signal,PCR}(A)&=R_{\rm opt\text{-}scale}(A)
+R_{\rm unit\text{-}excess}(A).
\label{eq:decomposition}
\end{align}
The first term is the minimum signal risk attainable by multiplying the selected
projector by one scalar; the second is the excess risk imposed by PCR's unit
multiplier.  The excess is strictly positive whenever $0<A<1$ and
$\lambda\ne\tau$.

\section{Simulations}
\label{sec:simulations}

Every Monte Carlo comparison is paired: methods within a
panel share the same design draw and, whenever response noise is drawn, the
same noise realization.  The heterogeneous rank experiment and the
fixed-aspect panels integrate response noise exactly, so they report
$\Rbar_X$ for each realized design.  The Gaussian simulations additionally
report dPCR with a training-spectrum plug-in correction, at the prespecified
rank in theorem comparisons and with explicitly stated anchors in exploratory
or validation sweeps.  The validation experiment instead
selects both rank and correction on an independent noisy validation sample.
Error bars are Monte Carlo
standard errors, and the reported metric is clean population prediction risk
\eqref{eq:pop-risk}, not noisy test error.  Collapsing sharp-tail and broad
Gaussian spiked-covariance designs are reported separately, so a broad MP bulk
is not treated as an isotropic floor.
The power-law settings instantiate the dimensionless regime
$\kappa_n\asymp1$ and $\gamma_n\to\infty$.  The computational archive
contains the complete simulation code, the per-experiment Monte Carlo summary
files, and a run manifest recording every setting.

Analytical curves display the proportional-limit formulas, while paired Monte
Carlo panels assess the sharp-floor and fixed-aspect predictions at the stated
finite sample sizes.  Table~\ref{tab:simulation-settings} lists the designs in the
order in which their figures are discussed below.

\begin{table}[!tbp]
\centering
\footnotesize
\caption{Simulation settings.  Here $E_H$ denotes
total head prediction energy and $a=T_1/n$; for a flat tail,
$a=m\lambda_{h+1}/n$.  Rows for the fixed-aspect figures use the local
$(\lambda,\tau)$ notation of Section~\ref{sec:fixed}.}
\label{tab:simulation-settings}
\begin{tabular}{>{\raggedright\arraybackslash}p{0.14\linewidth}
>{\raggedright\arraybackslash}p{0.29\linewidth}
>{\raggedright\arraybackslash}p{0.48\linewidth}}
\toprule
Figure & Design and purpose & Settings \\
\midrule
\ref{fig:matched}
& Rank-one spiked covariance limit across spike size
& $\gamma=4$, $\tau=1$, $E_H=1$; deterministic outlier/overlap formulas. \\
\ref{fig:gaussian-sharp}
& Full Gaussian sharp-floor spiked covariance experiment
& $n\in\{35,45,55,70,85,100,120,140\}$,
$m=\lceil n^{1.75}\rceil$,
$\lambda_{2,n}=n/m$ so $a=1$, $\lambda_{1,n}=E_H=1$,
$\sigma_\varepsilon=0.5$, $64$ replications; oracle and plug-in floors on
identical draws. \\
\ref{fig:heterogeneous-ranks}
& Gaussian multi-head, heterogeneous-tail rank sweep
& $n\in\{45,55,70,85,105,125,150,180\}$,
$m=\lceil n^{1.5}\rceil$, $a=1$, six head eigenvalues from $1.80$ to $.80$
with total $E_H=1$, and a nonflat quadratic tail profile normalized to
$T_1=n$.  Tail prediction energy is $E_T/E_H=.4n^{-1/2}$;
$\sigma_\varepsilon=.5$, $64$ design replications, ranks $0$ through $12$,
and exact conditional risk.  Plug-in dPCR uses $\widehat a_6$ and a $10\%$
relative denominator margin. \\
\ref{fig:fixed}
& Gaussian rank-one spiked covariance aspect-ratio sweep
& $n=100$, $\gamma\in\{1.25,1.5,2,2.75,4,5.5,8,12\}$, $\tau=1$,
$\lambda/(\gamma\tau)=2$, $\sigma_\varepsilon=0.5$, $64$ replications;
per-replication plug-in floor \eqref{eq:plugin-floor}. \\
\ref{fig:fixed-spike-sweep}
& Gaussian fixed-aspect spiked covariance spike-strength sweep
& $n=100$, $\gamma=4$, $\tau=1$,
$\lambda/(\gamma\tau)\in\{.80,.90,1,1.15,1.35,1.60,2,2.5,3\}$,
$\sigma_\varepsilon=0.5$, $64$ replications; per-replication plug-in floor
\eqref{eq:plugin-floor}. \\
\ref{fig:adaptive-validation}
& Joint validation selection in a Gaussian sharp-floor model
& Training and validation sizes $n\in\{45,55,70,85,105,125,150,180\}$,
$m=\lceil n^{1.5}\rceil$, $a=1$,
$(\lambda_1,\lambda_2,\lambda_3)=(1.6,1.2,0.8)$,
$E=(0.35,0.35,0.30)$,
$\sigma_\varepsilon=0.5$, $64$ replications; ranks $0$ through $12$ and five
training-spectrum correction levels. \\
\ref{fig:regime-transfer}
& Gaussian rank-one transfer across three tail regimes
& $n\in\{40,50,65,80,100,125,150,180\}$, $\lambda_{1,n}=E_H=1$,
$\sigma_\varepsilon=0.5$, $64$ replications.  Light:
$m=\lceil n^{1.5}\rceil$, $a_n=n^{-0.75}$; matched sharp:
$m=\lceil n^{1.5}\rceil$, $a_n=1$; broad/high-mass:
$m=\lceil1.25n\rceil$, $a_n=1$. \\
\bottomrule
\end{tabular}
\end{table}

\paragraph*{Matched-scale gains and the limitation of rank selection}
Figure~\ref{fig:matched} illustrates the pointwise rank-one formulas in
Theorem~\ref{thm:fixed-rank1}.  It shows that scalar correction is
most consequential when the predictive spike and the aggregate floor have
comparable scale, while the advantage narrows for a very strong spike.  The
uniform PCR lower bound in Theorem~\ref{cor:gaussian-head} strengthens this
pointwise message by comparing dPCR with every retained rank.

\paragraph*{Gaussian sharp-floor scaling}
Figure~\ref{fig:gaussian-sharp} uses the fully Gaussian head-and-tail design covered
by Corollary~\ref{cor:flat-tail}.  Here
$h=1=o(n)$ and the weak-tail dimension grows as $m\asymp n^{1.75}$, so the
relative operator error in \eqref{eq:flat-wishart-collapse} decreases at the
rate $n^{-0.375}$.  The growing dPCR--PCR separation is the finite-sample
pattern predicted by Corollary~\ref{cor:flat-tail}; the oracle-dPCR risk follows
the power-law order in \eqref{eq:power-law-illustration}.  The figure also
reports dPCR with a data-driven trimmed-mean correction and the
prespecified rank $h=1$ on the same draws; as predicted by
Corollary~\ref{cor:plugin-floor}, the plug-in markers are nearly indistinguishable
from the oracle-floor markers.

\begin{figure}[!tbp]
\centering
\begin{subfigure}[t]{0.485\linewidth}
\centering
\includegraphics[width=\linewidth]{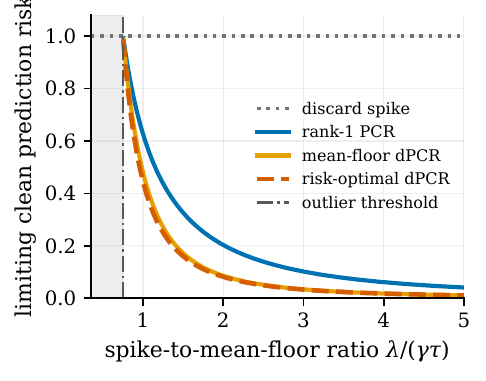}
\caption{Matched spike-to-floor theory.}
\label{fig:matched}
\end{subfigure}\hfill
\begin{subfigure}[t]{0.485\linewidth}
\centering
\includegraphics[width=\linewidth]{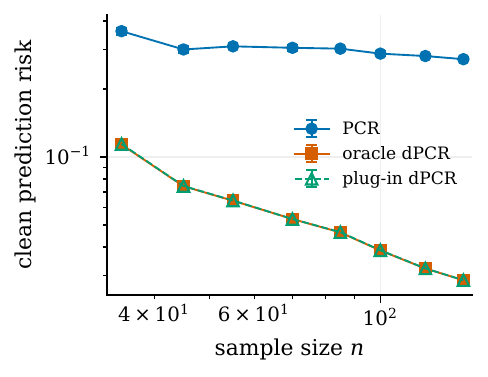}
\caption{Gaussian sharp-floor scaling.}
\label{fig:gaussian-sharp}
\end{subfigure}
\caption{Matched-scale denominator correction in two regimes.
\textbf{Left:} limiting rank-$1$ risk across the spike-to-mean-floor ratio;
the shaded region lies below the outlier threshold and curves use
Theorem~\ref{thm:fixed-rank1}.  \textbf{Right:} fully Gaussian rank-$1$ sharp-floor
simulation with $m=\lceil n^{1.75}\rceil$ and
$\lambda_2=n/m$, illustrating Corollary~\ref{cor:flat-tail}.  Lines connect paired Monte Carlo means and error
bars are Monte Carlo standard errors.  Open triangles use the same-sample
trimmed-mean floor \eqref{eq:plugin-floor}.}
\label{fig:matched-sharp}
\end{figure}

\paragraph*{Heterogeneous tails, multiple head components, and approximate
alignment.}
Figure~\ref{fig:heterogeneous-ranks} directly exercises the general random-design
result rather than its flat rank-one specialization.  There are six predictive
head coordinates, and the $m=\lceil n^{1.5}\rceil$ tail eigenvalues are
\begin{equation}
\lambda_{h+j}=\frac{n w_j}{\sum_{\ell=1}^m w_\ell},
\qquad
w_j=.35+1.30\left(1-\frac{j-1}{m-1}\right)^2,
\qquad 1\le j\le m.
\label{eq:heterogeneous-simulation-tail}
\end{equation}
Thus $a=T_1/n=1$, but the population tail is nonflat and has
$d_1\asymp d_2\asymp m$.  To test Corollary~\ref{cor:approx-alignment} as well, the
tail carries prediction energy $E_T/E_H=.4n^{-1/2}$, spread evenly in
prediction norm across its coordinates.

The left panel evaluates all thirteen ranks at $n=125$ using exact conditional
risk.  Ranks below six pay omitted-head loss; after all head coordinates enter,
ordinary PCR encounters the denominator-inflation barrier.  Oracle dPCR uses
$a=1$, while plug-in dPCR uses the single same-sample estimate
$\widehat a_6$ throughout the displayed sweep.  The theorem-covered point for
both corrected estimators is $r=h=6$; other ranks show how the same correction
behaves around that benchmark, subject to the stated denominator margin.  The
right panel makes the theorem comparison directly: it contrasts the minimum
conditional risk of ordinary PCR over ranks $0$ through $12$ with oracle and
plug-in dPCR at rank six.  Across all eight sample sizes, the plug-in curve
tracks the oracle-floor curve while both separate from the best PCR rank,
illustrating Theorem~\ref{cor:gaussian-head} and Corollaries~\ref{cor:plugin-floor} and \ref{cor:approx-alignment}.
For plug-in dPCR, the paired risk difference relative to best PCR changes from
$-0.0575$ (SE $0.0069$) at $n=45$ to $-0.1754$ (SE $0.0030$) at $n=180$;
rank-six oracle and plug-in candidates are admissible in all $64$ replications
at every displayed size.

\begin{figure}[!tbp]
\centering
\includegraphics[width=\linewidth]{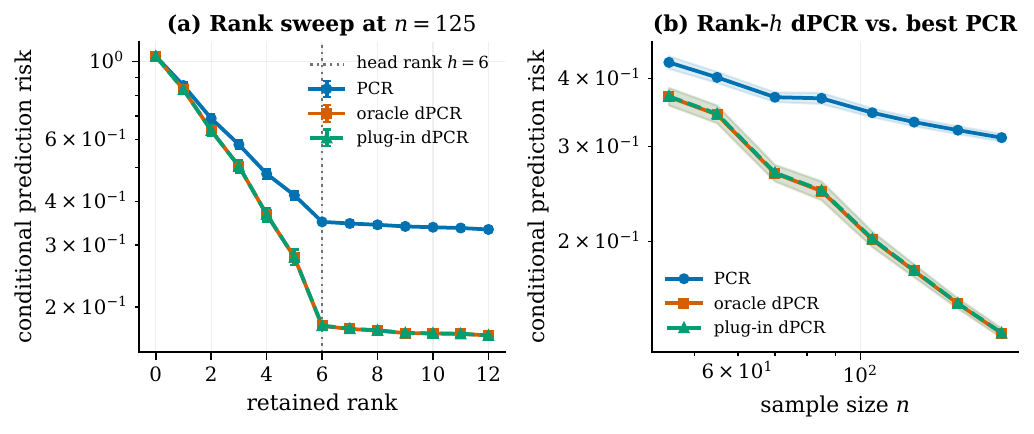}
\caption{Multi-head heterogeneous Gaussian sharp-floor simulation with
vanishing tail prediction energy.  \textbf{Left:} exact conditional population risk over
the dense rank grid $r=0,\ldots,12$ at $n=125$; the dotted line marks the six
population head coordinates.  \textbf{Right:} best ordinary-PCR conditional
risk over the same grid versus rank-$6$ oracle and plug-in dPCR at eight sample
sizes.  The nonflat tail is given
by \eqref{eq:heterogeneous-simulation-tail}, has $a=1$ and
$m=\lceil n^{1.5}\rceil$, and carries $E_T/E_H=.4n^{-1/2}$.  Methods share
every design; markers are paired Monte Carlo means and bands or error bars are
one Monte Carlo standard error across $64$ designs.  Oracle dPCR subtracts $a$,
while plug-in dPCR uses $\widehat a_6$; corrected candidates must retain at
least $10\%$ of the uncorrected boundary eigenvalue as denominator margin.}
\label{fig:heterogeneous-ranks}
\end{figure}

\paragraph*{Fixed-aspect outlier and overlap effects}
Both panels of Figure~\ref{fig:fixed-aspect-simulations} illustrate
Theorem~\ref{thm:fixed-rank1}.  The left panel varies the tail aspect ratio while
maintaining a separated spike, and the right panel varies spike strength at a
fixed aspect ratio.  Agreement between the finite-sample markers and the
outlier/overlap curves checks the two ingredients of the theorem: the sample
outlier determines the corrected denominator, while eigenvector overlap
determines how that multiplier trades head recovery against tail loading.  The
plug-in dPCR series uses a data-driven correction at the prespecified rank:
each replication subtracts the trimmed-mean floor
\eqref{eq:plugin-floor}, whose fixed-aspect limit is mean-bulk
subtraction.  Its visible distinction from the risk-optimal correction is the
property emphasized by the fixed-aspect theory: a broad MP bulk cannot in
general be summarized by subtracting its mean.  These panels concern the
pointwise rank-one limit in Theorem~\ref{thm:fixed-rank1}.

\begin{figure}[!tbp]
\centering
\begin{subfigure}[t]{0.485\linewidth}
\centering
\includegraphics[width=\linewidth]{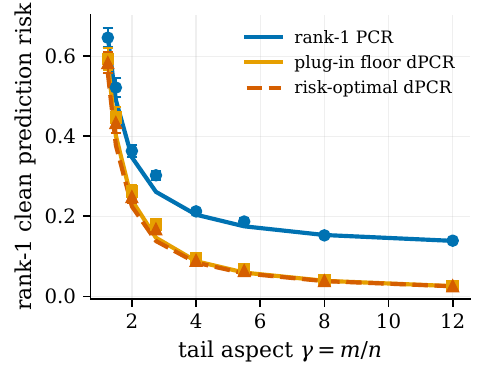}
\caption{Aspect-ratio sweep.}
\label{fig:fixed}
\end{subfigure}\hfill
\begin{subfigure}[t]{0.485\linewidth}
\centering
\includegraphics[width=\linewidth]{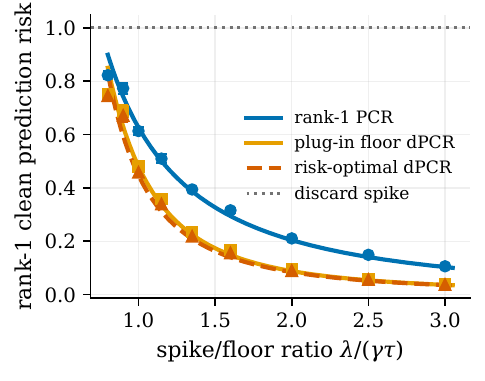}
\caption{Spike-strength sweep.}
\label{fig:fixed-spike-sweep}
\end{subfigure}
\caption{Fixed-aspect Gaussian rank-one simulations and their
outlier/overlap approximations.  \textbf{Left:} risk across tail aspect ratios
for a separated spike.  \textbf{Right:} risk across spike-to-mean-floor ratios
at $(n,\gamma,\tau)=(100,4,1)$.  Curves are finite-$n$ approximations and
markers are paired Monte Carlo means.  The plug-in series subtracts the
trimmed-mean floor \eqref{eq:plugin-floor} computed from each realized
spectrum; its curve is the mean-floor limit.  Plug-in and risk-optimal
corrections both improve rank-$1$ PCR in the displayed regimes, while their
distinction confirms that a broad MP bulk is not an isotropic floor.}
\label{fig:fixed-aspect-simulations}
\end{figure}

\paragraph*{Joint validation selection of rank and correction}
Figure~\ref{fig:adaptive-validation} implements the validation procedure in
Remark~\ref{rem:validation} without supplying either $h$ or $a$ to the selected
estimator.  The training and validation samples are independent and have the
same size $n$.  For every candidate rank $1\le r\le12$, the training spectrum
defines the rank-specific anchor
$(n-r)^{-1}\sum_{i>r}\widehat\mu_i$.  The correction grid multiplies this
anchor by $0$, $0.5$, $0.75$, $1$, or $1.25$, after removing candidates with
$\widehat\mu_r-a_0<0.1\widehat\mu_r$.  Rank zero is also included.  The pair
$(\widehat r,\widehat a)$ minimizes noisy prediction error on the validation
sample; only after selection do we evaluate clean population prediction risk.

The left panel shows that joint validation retains most of the gain over
validation-selected PCR and tracks the known-rank plug-in and oracle-floor
benchmarks as $n$ grows.  Across the eight displayed sizes, the average selected
correction is between $1.01a$ and $1.07a$.  The paired dPCR-minus-PCR risk
difference ranges from $-0.136$ (SE $0.006$) at $n=45$ to $-0.203$ (SE $0.004$)
at $n=180$.  All candidates satisfy the $10\%$ denominator margin in these
sharp-floor simulations.  The right panel reports the selected
correction together with the selected rank normalized by the largest candidate
rank.  The rank need not equal the population head dimension: validation targets
prediction risk, and at finite samples it may retain additional empirical
directions that carry leaked head signal.

\begin{figure}[!tbp]
\centering
\includegraphics[width=\linewidth]{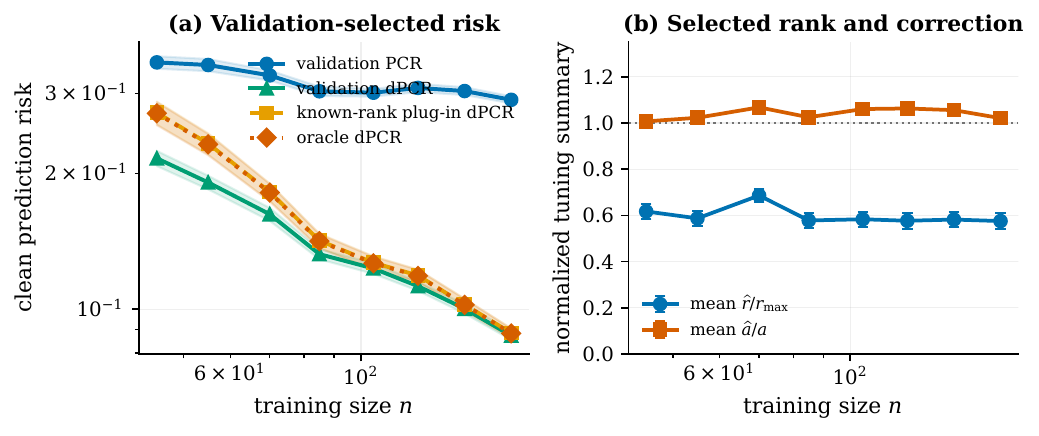}
\caption{Fully validation-adaptive dPCR in a three-spike Gaussian sharp-floor
model with $m=\lceil n^{1.5}\rceil$.  \textbf{Left:} clean population risk of
validation-selected PCR and jointly selected dPCR, with known-rank plug-in and
oracle dPCR benchmarks.  Training and validation samples are independent, and
selection uses noisy validation MSE.  Bands show one Monte Carlo standard error.
\textbf{Right:} average selected rank, normalized by $r_{\max}=12$, and average
selected correction, normalized by the true floor only for reporting.  The
validation procedure itself receives neither $h$ nor $a$.}
\label{fig:adaptive-validation}
\end{figure}

\paragraph*{Transfer across tail regimes}
Figure~\ref{fig:regime-transfer} changes the tail scale and tail aspect while holding
the rank-one population spike and prediction energy fixed.  This makes the
three qualitative regimes in the theory directly comparable.  In the light
sequence, $a_n=n^{-0.75}\to0$, so the squared attenuation bias is lower order
than the shared $n^{-1}$ response-noise scale.  Both methods become accurate
and scalar correction offers no first-order advantage.  In the matched sharp
sequence, $a_n=\lambda_{1,n}=1$ while
$m/n\asymp n^{1/2}\to\infty$; the empirical tail Gram concentrates around a
nonnegligible scalar floor, producing the separation described by
Corollary~\ref{cor:sharp-separation}.  In the broad/high-squared-mass sequence,
$m/n\to1.25$ and $a_n=1$.  The Marchenko--Pastur bulk has nonvanishing relative
width and scalar mean-floor subtraction cannot undo its directional
distortion; both rank-one procedures remain close to the risk of discarding
the spike.  Here the population spike lies below the separated-spike threshold,
so this panel is a qualitative broad-bulk diagnostic rather than an
illustration of Theorem~\ref{thm:fixed-rank1}.

\begin{figure}[!tbp]
\centering
\includegraphics[width=\linewidth]{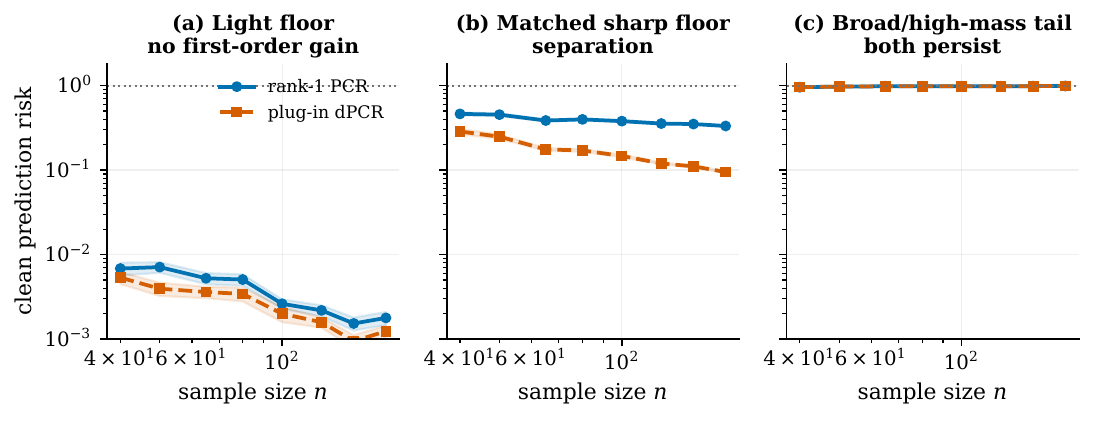}
\caption{Regime transfer in paired Gaussian rank-one spiked covariance
simulations.  Each panel uses eight sample sizes and reports clean population
prediction risk; bands are one Monte Carlo standard error and the dotted line
is the risk of discarding the unit-energy spike.  \textbf{Left:} a vanishing
aggregate floor makes both risks vanish.  \textbf{Center:} a nonnegligible but
relatively sharp floor produces persistent PCR bias and declining plug-in dPCR
risk.  \textbf{Right:} a proportional, broad tail prevents scalar mean-floor
subtraction from recovering the sharp-floor behavior, and both risks persist.}
\label{fig:regime-transfer}
\end{figure}

\section{Discussion and limitations}
\label{sec:discussion}

\paragraph*{What the sharp random-design theorem establishes}
The all-rank PCR barrier is a prediction-risk dichotomy: excluding predictive
head directions pays omission loss, while retaining them leaves
floor-induced attenuation.  The dPCR bound shows that removing the leading
head bias is cheap when the same tail has high enough effective dimension for
both its Gram contribution and its squared-tail clean-risk readout to
concentrate.  Thus a tail can have enough trace to create a floor matched to
the head at any absolute scale while having little enough squared mass for
dPCR risk to vanish.

Projected negative-ridge algebra and ordinary cluster perturbation give a
complete total-risk theorem.  A global first-order inverse error contributes
$n/d_1$ after squaring, the same order as the baseline clean-tail leakage.
Consequently, the sharper decomposition $h/d_1+(n/d_1)^2$ is unnecessary for
risk separation.  Nonflat tails still require largest-coordinate control and a
second effective dimension for the squared-tail clean-risk readout.  Because
every sub-cluster sample eigenvalue is trapped within the tail operator width
of the floor, the trimmed-mean plug-in floor attains the oracle rate, and a
median-threshold rule recovers the rank in the matched regime.  The governing
quantities are the head-to-floor ratio, the two tail effective
dimensions, the largest tail coordinates, and head prediction SNR, not a
particular absolute covariance scale.  The matched Gaussian sharpness
condition implies the boundary relation
$\lambda_{h_n+1,n}/\lambda_{h_n,n}\to0$, while the proof itself uses the
corrected empirical margin in \eqref{eq:corrected-cluster-margin}.
Remaining extensions include
designs beyond Gaussian random design---where population diagonalization need
not make the head and tail independent---and a validation theory for jointly
selecting $(a_0,r)$ outside the sharp regime.

\paragraph*{Scalar versus nonlinear denominator correction}
The fixed-aspect result is both positive and limiting.  The risk-optimal scalar
correction strictly improves rank-$1$ PCR on separated
spikes, but the mean floor is not the exact prediction-optimal correction.
The derivation balances the outlier map and eigenvector overlap against tail
prediction weight; in the flat one-spike model, the resulting optimal corrected
denominator simplifies exactly to the population spike $\lambda$.  A
component-wise MP-corrected estimator is therefore a benchmark for broad bulks
rather than the main method in the sharp-floor paper.

\paragraph*{Novelty boundary}
Building on the rank-one flat-tail prototype of
\citet{dicker2013oneshot}, we develop a growing-sample prediction theory with
several predictive components and tail spectra satisfying the sharp-floor
conditions, including both flat and nonflat cases.  The theory covers
rank-uniform comparison, clean prediction leakage, same-sample floor
estimation, approximate alignment, and the broad-bulk boundary.  The detailed
comparison with adjacent methods is given in Section~\ref{sec:related}.

\section{Conclusion}

PCR regularizes prediction through rank alone.  An aggregate tail floor can
inflate the denominators of retained terms and create bias that survives every
rank choice, including an oracle one.  De-floored PCR adds denominator
correction as a second regularization control.  In
Gaussian sharp matched-shell regimes covering flat and nonflat tails, its risk
ratio relative to the best PCR rank vanishes at arbitrary absolute covariance
scale, with a same-sample trimmed-mean floor attaining the oracle rate; the
same separation survives when tail prediction energy is $o(E_H)$.  Separate pointwise
fixed-aspect results show where scalar correction stops being exact.  The
projected negative-ridge proof also clarifies why
effective tail dimension controls both floor sharpness and the cost of clean
prediction leakage.

\clearpage
\appendix
% !TEX root = defloored_pcr.tex
\section{Proofs and supporting derivations}
\label{app:proofs}

Throughout the appendices, probabilities are over the random design unless an
expectation is explicitly indexed by the response noise.  The quantity
$\Rbar_X$ already averages over response noise conditional on the realized
design.  Constants denoted by $C$ are universal, while $C_{C_H}$ may depend
only on the fixed head-shell constant.  Each high-probability result is first
proved on a deterministic design event and then combined with the stated
Gaussian concentration bounds.

\subsection{Projected negative-ridge random-design proof}
\label{app:sharp-support}

\subsubsection{Deterministic comparison statements}

The Gaussian sharp-floor theorem is proved through the following algebraic
realized-head benchmark.  Gaussian head concentration verifies the benchmark
and transfers its operator bounds to the population prediction metric.

\begin{definition}[Deterministic realized-head benchmark]
\label{ass:conditioned-head}
Fix $X_H=\sqrt n\,UD_H^{1/2}$, where $U^\top U=I_h$ and
$D_H=\diag(s_1,\ldots,s_h)$ with $s_1\ge\cdots\ge s_h>0$.
For the deterministic comparison take $\Sigma_H=D_H$ and assume
$s_1/s_h\le C_H$ for a fixed $C_H$.
\end{definition}
The proof of Theorem 4.4 of the main paper separately transfers these operator
bounds to the population metric $\Lambda_H$; it does not condition on an
identity between $D_H$ and $\Lambda_H$.

Under Definition~\ref{ass:conditioned-head}, let $K_H=UD_HU^\top$ and put
\begin{equation}
\kappa:=\frac{s_h}{a},
\qquad
E_H:=\sum_{j=1}^hs_j(\beta_j^\star)^2,
\qquad
V_H:=\frac{\sigma_\varepsilon^2h}{n}.
\label{eq:heterogeneous-head-scales}
\end{equation}

\begin{proposition}[Projected inverse and conditional risk]
\label{prop:projected-inverse-risk}
Under Assumption 4.1 of the main paper and the benchmark in
Definition~\ref{ass:conditioned-head}, write
$\Delta=K_T^{(1)}-aI_n$, $P=UU^\top$, and
$L=K-aI_n=K_H+\Delta$.  Let $\widehat P$ be the projector onto the top $h$
eigenvalues of $L$ and define
\begin{equation}
\widehat R
=\sum_{i=1}^h\frac{1}{\widehat\mu_i-a}
\widehat u_i\widehat u_i^\top,
\qquad
R_0=K_H^+=UD_H^{-1}U^\top.
\label{eq:cluster-inverses}
\end{equation}
Assume $\beta_T^\star=0$.
The feature estimator generated by this sample-space operator is exactly
$X^\top\widehat R y/n=\widehat\beta_{\dPCR}(a,h)$.
If $\delta:=\norm{\Delta}_{\op}\le s_h/4$, then
\begin{equation}
\norm{\widehat P-P}_{\op}
\le C\frac{\delta}{s_h},
\qquad
\norm{\widehat R}_{\op}\le\frac{2}{s_h},
\qquad
\norm{\widehat R-R_0}_{\op}
\le C\frac{\delta}{s_h^2}.
\label{eq:global-cluster-inverse}
\end{equation}
In particular, $(a,h)$ is admissible whenever $g_n\le s_h/2$.

For any symmetric design-measurable sample-space operator $R=R(X)$, let
$\widehat\beta_R=X^\top Ry/n$ and $y_0=X_H\beta_H^\star$.  With
$K_H^{(2)}=UD_H^2U^\top$, its conditional prediction risk is exactly
\begin{align}
\Rbar_X(\widehat\beta_R)
={}&
\norm{
D_H^{1/2}
\left(\frac{X_H^\top RX_H}{n}-I_h\right)
\beta_H^\star}^2
+\frac1n y_0^\top RK_T^{(2)}Ry_0
\nonumber\\
&+\frac{\sigma_\varepsilon^2}{n}
\tr(RK_H^{(2)}R)
+\frac{\sigma_\varepsilon^2}{n}
\tr(RK_T^{(2)}R).
\label{eq:operator-risk-identity}
\end{align}
Consequently,
\begin{equation}
\Rbar_X(\widehat\beta_{\widehat R})
\le C_{C_H}\left[
V_H+(E_H+V_H)\left\{
\frac{\delta^2}{s_h^2}
+\frac{\norm{K_T^{(2)}}_{\op}}{s_h^2}
\right\}
\right].
\label{eq:deterministic-random-transfer}
\end{equation}
\end{proposition}

The global perturbation term is first order in the fitted operator and hence
quadratic in prediction risk.  Its scale is already no larger than the
baseline clean-tail leakage term, so a second-order Feshbach expansion is not
needed for the total-risk rate.

\begin{lemma}[Weighted Gaussian tail concentration]
\label{lem:weighted-tail-concentration}
Under Assumption 4.1 of the main paper, there is a universal constant $C$ such that,
for every $t\ge0$, with
probability at least $1-4e^{-t}$,
\begin{equation}
\norm{K_T^{(1)}-aI_n}_{\op}\le Ca q_1(t),
\qquad
\norm{K_T^{(2)}}_{\op}\le b_T\{1+Cq_2(t)\}.
\label{eq:weighted-tail-concentration}
\end{equation}
\end{lemma}

\begin{proposition}[Conditional sharp-floor risk]
\label{thm:sharp-dpcr}
Under Assumption 4.1 of the main paper and the benchmark in
Definition~\ref{ass:conditioned-head}, suppose
$\beta_T^\star=0$ and
\begin{equation}
q_1(t)\le c\min(\kappa,1)
\label{eq:sharp-smallness}
\end{equation}
for a sufficiently small universal constant $c$.  If $g_n\le s_h/2$,
then with probability at least $1-4e^{-t}$, oracle dPCR with
$(a_0,r)=(a,h)$ is admissible and
\begin{equation}
\Rbar_X\{\widehat\beta_{\dPCR}(a,h)\}
\le C_{C_H}\left[
V_H+\frac{E_H+V_H}{\kappa^2}
\left\{
q_1(t)^2+\frac{n}{d_1}\{1+q_2(t)\}
\right\}
\right].
\label{eq:sharp-dpcr-bound}
\end{equation}
The dimensionless leakage factors depend on ratios and effective dimensions.
In particular, the result imposes no constant-order assumption on $a$, $s_h$,
or $s_1$.
\end{proposition}

The two terms inside braces have different origins.  The term $q_1(t)^2$
controls perturbation of the projected inverse, whereas
\begin{equation}
\frac{b_T}{s_h^2}
=\frac{T_2}{ns_h^2}
=\frac{n}{d_1\kappa^2}
\label{eq:heterogeneous-leakage}
\end{equation}
is the baseline clean-tail leakage.  Thus the total rate is $n/d_1$ in a
matched sharp regime.  A refined head-sandwiched analysis could decompose the
perturbation alone into $h/d_1+(n/d_1)^2$, but that refinement cannot improve
the total rate because \eqref{eq:heterogeneous-leakage} is already $n/d_1$.

\begin{proposition}[Conditional perturbed-floor PCR barrier]
\label{thm:sharp-pcr-lower}
Under Assumption 4.1 of the main paper and the benchmark in
Definition~\ref{ass:conditioned-head}, assume
$\beta_T^\star=0$ and put
$\delta=\norm{K_T^{(1)}-aI_n}_{\op}$.  On the event
$\delta<a$,
\begin{equation}
\inf_{0\le r\le n}
\Rbar_X\{\widehat\beta_{\PCR}(r)\}
\ge
\frac{s_h}{s_1}
\left(
\frac{a-\delta}{s_1+a-\delta}
\right)^2E_H.
\label{eq:sharp-pcr-lower}
\end{equation}
If \eqref{eq:sharp-smallness} holds, the same event as in
Proposition~\ref{thm:sharp-dpcr} also gives
\begin{equation}
\inf_r\Rbar_X\{\widehat\beta_{\PCR}(r)\}
\ge c_{C_H}\frac{E_H}{(1+C_H\kappa)^2}.
\label{eq:sharp-pcr-scale-free-lower}
\end{equation}
No eigengaps inside the head cluster are assumed.  Only the corrected
head--residual cluster margin used for inverse stability is required.
\end{proposition}

\subsubsection{Proofs of the deterministic statements and Gaussian transfer}

For a symmetric matrix $A$, let $\operatorname{eig}_i(A)$ denote its $i$th
largest eigenvalue, counted with multiplicity.  This generic matrix-eigenvalue
map is distinct from the full-sample eigenvalues $\widehat\mu_i$.

\begin{proof}[Proof of Proposition~\ref{prop:projected-inverse-risk}]
The matrices $K$ and $L=K-aI_n$ have the same eigenvectors, and the
eigenvalue of $L$ associated with $\widehat u_i$ is
$\widehat\mu_i-a$.  For every positive empirical eigenvalue,
\begin{equation*}
\widehat v_i=\frac{X^\top\widehat u_i}{\sqrt{n\widehat\mu_i}},
\qquad
\widehat v_i^\top X^\top y/n
=\sqrt{\widehat\mu_i/n}\,\widehat u_i^\top y.
\end{equation*}
Multiplying these two identities shows term by term that
\begin{equation*}
\widehat\beta_{\dPCR}(a,h)
=\sum_{i=1}^h
\frac{X^\top\widehat u_i\widehat u_i^\top y}
{n(\widehat\mu_i-a)}
=\frac{X^\top\widehat R y}{n}.
\end{equation*}

We next establish the deterministic spectral bounds.  Weyl's inequality gives
\begin{equation}
\operatorname{eig}_h(K_H+\Delta)\ge s_h-\delta\ge\frac34s_h,
\qquad
\operatorname{eig}_{h+1}(K_H+\Delta)\le\delta\le\frac14s_h.
\label{eq:app-cluster-gap}
\end{equation}
Thus the top-$h$ cluster is separated from the remaining spectrum by at least
$s_h/2$.  The Davis--Kahan sin-theta theorem, applied to the full
$h$-dimensional cluster rather than to its individual eigenvectors, gives
\begin{equation}
\norm{\widehat P-P}_{\op}
\le C\frac{\delta}{s_h}.
\label{eq:app-projector-bound}
\end{equation}
The first inequality in \eqref{eq:app-cluster-gap} also gives
$\norm{\widehat R}_{\op}\le4/(3s_h)\le2/s_h$.

It remains to compare the two inverse operators.  Since
$\widehat R(K_H+\Delta)=\widehat P$, $K_HR_0=P$, and $PR_0=R_0$, multiplication on
the right by $R_0$ gives
\begin{equation}
\widehat RP-R_0
=(\widehat P-P)R_0-\widehat R\Delta R_0.
\label{eq:app-inverse-identity}
\end{equation}
Moreover, $\widehat R=\widehat R\widehat P$, so
\begin{equation}
\widehat R-R_0
=\widehat R(\widehat P-P)+(\widehat RP-R_0).
\label{eq:app-inverse-split}
\end{equation}
Combining \eqref{eq:app-projector-bound}--\eqref{eq:app-inverse-split} with
$\norm{R_0}_{\op}=1/s_h$ proves
\begin{equation}
\norm{\widehat R-R_0}_{\op}
\le C\frac{\delta}{s_h^2}.
\label{eq:app-global-inverse-bound}
\end{equation}
No step uses an eigengap within $[s_h,s_1]$.  Also
$\widehat\mu_h-a=\operatorname{eig}_h(K_H+\Delta)\ge3s_h/4$, so $g_n\le s_h/2$
implies admissibility.

We next derive the risk identity.  For a symmetric design-measurable $R=R(X)$,
\begin{equation}
\widehat\beta_{R,H}
=\frac{X_H^\top R(y_0+\varepsilon)}{n},
\qquad
\widehat\beta_{R,T}
=\frac{X_T^\top R(y_0+\varepsilon)}{n}.
\label{eq:app-feature-blocks}
\end{equation}
Under the exact-alignment benchmark, the signal lies in the head and the
population covariance is block diagonal.
After expanding the two population-covariance norms, the signal--noise cross
terms vanish under $\E_\varepsilon(\varepsilon\varepsilon^\top)
=\sigma_\varepsilon^2I_n$.  The noiseless head term is the first term in
\eqref{eq:operator-risk-identity}.  For the noiseless tail term,
\begin{equation}
\norm{\Lambda_T^{1/2}X_T^\top Ry_0/n}^2
=\frac1n y_0^\top RK_T^{(2)}Ry_0.
\label{eq:app-tail-signal-identity}
\end{equation}
The head-noise trace equals
\begin{equation}
\frac{\sigma_\varepsilon^2}{n^2}
\tr(RX_HD_HX_H^\top R)
=\frac{\sigma_\varepsilon^2}{n}\tr(RK_H^{(2)}R),
\label{eq:app-head-noise-identity}
\end{equation}
and the same calculation with $(X_T,\Lambda_T)$ gives the last term in
\eqref{eq:operator-risk-identity}.  This proves the identity with every
expectation conditional on the realized design.

Finally, take $R=\widehat R$ and write $\Xi=\widehat R-R_0$.  The noiseless head
map satisfies
\begin{equation}
\frac{X_H^\top\widehat RX_H}{n}-I_h
=D_H^{1/2}U^\top \Xi UD_H^{1/2}.
\label{eq:app-head-map-global}
\end{equation}
Using $\norm{\beta_H^\star}^2\le E_H/s_h$ and
$s_1/s_h\le C_H$, its prediction loss is at most
\begin{equation}
C_{C_H}E_H\frac{\delta^2}{s_h^2}.
\label{eq:app-head-bias-global}
\end{equation}
Also $\norm{y_0}^2/n=E_H$, and therefore
\begin{equation}
\frac1n y_0^\top\widehat R K_T^{(2)}\widehat R y_0
\le
\frac{4\norm{K_T^{(2)}}_{\op}}{s_h^2}E_H.
\label{eq:app-tail-signal-global}
\end{equation}
For head noise, $\tr(R_0K_H^{(2)}R_0)=h$ and
$\norm{A+B}_F^2\le2\norm A_F^2+2\norm B_F^2$ give
\begin{equation}
\frac{\sigma_\varepsilon^2}{n}
\tr(\widehat R K_H^{(2)}\widehat R)
\le C_{C_H}V_H
+ C_{C_H}V_H\frac{\delta^2}{s_h^2}.
\label{eq:app-head-noise-global}
\end{equation}
Because $\widehat R$ has rank $h$,
\begin{equation}
\frac{\sigma_\varepsilon^2}{n}
\tr(\widehat R K_T^{(2)}\widehat R)
\le
4V_H\frac{\norm{K_T^{(2)}}_{\op}}{s_h^2}.
\label{eq:app-tail-noise-global}
\end{equation}
Adding \eqref{eq:app-head-bias-global}--\eqref{eq:app-tail-noise-global}
proves \eqref{eq:deterministic-random-transfer}.
\end{proof}

\begin{proof}[Proof of Lemma~\ref{lem:weighted-tail-concentration}]
We use the same scalar argument for both tail matrices.  Let $w_j\ge0$,
$W_1=\sum_jw_j$, $W_2=\sum_jw_j^2$, and $w_{\max}=\max_jw_j$.  For a fixed
unit vector $u$, the variables $u^\top z_j$ are independent $N(0,1)$, so the
moment-generating function of a centered chi-square variable implies
\begin{equation}
\Pp\left(
\left|\sum_jw_j\{(u^\top z_j)^2-1\}\right|
>C\{\sqrt{W_2s}+w_{\max}s\}
\right)
\le2e^{-s}
\label{eq:app-weighted-chi-square}
\end{equation}
for every $s\ge0$.  To pass from a fixed direction to operator norm, take a
$1/4$-net $\mathcal N$ of the Euclidean unit sphere with
$|\mathcal N|\le9^n$.  For every symmetric matrix $M$,
\begin{equation}
\norm M_{\op}\le2\max_{u\in\mathcal N}|u^\top Mu|.
\label{eq:app-net-bound}
\end{equation}
Apply \eqref{eq:app-weighted-chi-square} on the net with $s=C_0(n+t)$ and
choose the universal constant $C_0$ large enough to absorb $9^n$.  After
division by $n$, this yields, with probability at least $1-2e^{-t}$,
\begin{equation}
\left\|
\frac1n\sum_jw_jz_jz_j^\top-\frac{W_1}{n}I_n
\right\|_{\op}
\le\frac{C}{n}
\left\{\sqrt{W_2(n+t)}+w_{\max}(n+t)\right\}.
\label{eq:app-weighted-wishart-general}
\end{equation}
With $w_j=\lambda_j$, the right-hand side is $Ca q_1(t)$.  With
$w_j=\lambda_j^2$, it is $Cb_Tq_2(t)$, so
$\norm{K_T^{(2)}}_{\op}\le b_T\{1+Cq_2(t)\}$.  A union bound over the two
events proves \eqref{eq:weighted-tail-concentration}.
\end{proof}

\begin{proof}[Proof of Proposition~\ref{thm:sharp-dpcr}]
On the event in Lemma~\ref{lem:weighted-tail-concentration},
\begin{equation}
\frac{\delta^2}{s_h^2}
\le \frac{Cq_1(t)^2}{\kappa^2},
\qquad
\frac{\norm{K_T^{(2)}}_{\op}}{s_h^2}
\le
\frac{n}{d_1\kappa^2}\{1+Cq_2(t)\}.
\label{eq:app-random-substitution}
\end{equation}
Indeed, $b_T=T_2/n=a^2n/d_1$.  The smallness condition
\eqref{eq:sharp-smallness}, with the universal constant chosen below those in
Lemma~\ref{lem:weighted-tail-concentration}, gives
$\delta\le\min(s_h/4,a/2)$.  Thus
Proposition~\ref{prop:projected-inverse-risk} applies.  Substitution of
\eqref{eq:app-random-substitution} into
\eqref{eq:deterministic-random-transfer} proves
\eqref{eq:sharp-dpcr-bound} and admissibility on an event of probability at
least $1-4e^{-t}$.
\end{proof}

\begin{proof}[Proof of Proposition~\ref{thm:sharp-pcr-lower}]
Let $G_r=\sum_{i\le r}\widehat\mu_i^{-1}
\widehat u_i\widehat u_i^\top$ be the sample-space PCR inverse.  Since
$m\ge n$ and all tail weights are positive, $K_T^{(1)}$ and hence $K$ are
positive definite almost surely.  Therefore
$0\preceq G_r\preceq K^{-1}$.  On
$K_T^{(1)}\succeq(a-\delta)I_n$, the noiseless head map
$M_r=X_H^\top G_rX_H/n$ satisfies
\begin{equation}
0\preceq M_r
\preceq
X_H^\top\{K_H+(a-\delta)I_n\}^{-1}X_H/n
=\diag\left\{
\frac{s_j}{s_j+a-\delta}
\right\}_{j=1}^h.
\label{eq:app-pcr-order}
\end{equation}
Thus $\norm{M_r}_{\op}\le s_1/(s_1+a-\delta)<1$.  The head part
of the conditional squared bias is at least
\begin{align}
\norm{(I_h-M_r)\beta_H^\star}_{D_H}^2
&\ge
s_h\{1-\norm{M_r}_{\op}\}^2
\norm{\beta_H^\star}^2
\nonumber\\
&\ge
\frac{s_h}{s_1}
\left(\frac{a-\delta}{s_1+a-\delta}\right)^2E_H.
\label{eq:app-pcr-bias-lower}
\end{align}
All remaining prediction-risk terms are nonnegative, so
\eqref{eq:sharp-pcr-lower} holds uniformly over $r$.  Under
\eqref{eq:sharp-smallness}, the concentration event used in
Proposition~\ref{thm:sharp-dpcr} has $\delta\le a/2$.  Since
$s_1/s_h\le C_H$ and $s_h/a=\kappa$,
\eqref{eq:app-pcr-bias-lower} is at least the right-hand side of
\eqref{eq:sharp-pcr-scale-free-lower}.
\end{proof}

\begin{lemma}[Transfer from the realized head to the population metric]
\label{lem:random-head-transfer}
Let $X_H\in\R^{n\times h}$ have full column rank, put
$S_H=X_H^\top X_H/n$ and $K_H=X_HX_H^\top/n$, and let
$\Lambda_H$ be positive definite.  Suppose
\begin{equation}
\frac12\Lambda_H\preceq S_H\preceq\frac32\Lambda_H,
\qquad
\frac{\lambda_1}{\lambda_h}\le C_H,
\label{eq:app-transfer-event}
\end{equation}
where $\lambda_1$ and $\lambda_h$ are the largest and smallest eigenvalues of
$\Lambda_H$.  Write $s_1\ge\cdots\ge s_h$ for the eigenvalues of $S_H$,
$R_0=K_H^+$, and
\begin{equation*}
K_{H,\Lambda}^{(2)}=\frac{X_H\Lambda_HX_H^\top}{n}.
\end{equation*}
Then the following statements hold.

\begin{enumerate}[label=(\roman*),leftmargin=2.2em]
\item $s_h\ge\lambda_h/2$, $s_1\le3\lambda_1/2$, and hence
$s_1/s_h\le3C_H$.
\item For $y_0=X_H\beta_H^\star$ and
$E_H=\beta_H^{\star\top}\Lambda_H\beta_H^\star$,
\begin{equation}
\frac{\norm{y_0}^2}{n}\le\frac32E_H,
\qquad
\tr(R_0K_{H,\Lambda}^{(2)}R_0)
=\tr(\Lambda_HS_H^{-1})\le2h.
\label{eq:app-transfer-energy-variance}
\end{equation}
\item For every symmetric $\Xi\in\R^{n\times n}$,
\begin{equation}
\left\|
\Lambda_H^{1/2}\frac{X_H^\top\Xi X_H}{n}
\Lambda_H^{-1/2}
\right\|_{\op}
\le C_{C_H}\lambda_h\norm{\Xi}_{\op}.
\label{eq:app-transfer-bias}
\end{equation}
If, in addition, $R=R_0+\Xi$ has rank at most $h$, then
\begin{equation}
\frac{\sigma_\varepsilon^2}{n}
\tr(RK_{H,\Lambda}^{(2)}R)
\le C V_H+C_{C_H}V_H\lambda_h^2\norm{\Xi}_{\op}^2.
\label{eq:app-transfer-head-noise}
\end{equation}
Moreover,
\begin{equation}
\frac{\sigma_\varepsilon^2}{n}\tr(RK_T^{(2)}R)
\le V_H\norm{K_T^{(2)}}_{\op}\norm R_{\op}^2.
\label{eq:app-transfer-tail-noise}
\end{equation}
\item Let $c_0>0$ and let $G$ be any symmetric operator satisfying
$0\preceq G\preceq\{K_H+c_0I_n\}^{-1}$.  For
$M=X_H^\top GX_H/n$,
\begin{equation}
0\preceq M\preceq S_H(S_H+c_0I_h)^{-1},
\qquad
\norm M_{\op}\le\frac{s_1}{s_1+c_0},
\label{eq:app-transfer-pcr-map}
\end{equation}
and therefore
\begin{equation}
\norm{(I_h-M)\beta_H^\star}_{\Lambda_H}^2
\ge \frac1{C_H}
\left(\frac{c_0}{s_1+c_0}\right)^2E_H.
\label{eq:app-transfer-pcr-bias}
\end{equation}
\end{enumerate}
\end{lemma}

\begin{proof}
Part (i) follows from the min--max characterization applied to
\eqref{eq:app-transfer-event}.  The first identity in part (ii) is
$\norm{y_0}^2/n=\beta_H^{\star\top}S_H\beta_H^\star$.  For the second, take a
thin singular-value decomposition
\begin{equation*}
X_H=\sqrt n\,UD_H^{1/2}V^\top,
\qquad
S_H=VD_HV^\top,
\qquad
R_0=UD_H^{-1}U^\top.
\end{equation*}
Direct multiplication and cyclicity of the trace give
\begin{equation*}
\tr(R_0K_{H,\Lambda}^{(2)}R_0)
=\tr(D_H^{-1}V^\top\Lambda_HV)
=\tr(\Lambda_HS_H^{-1}).
\end{equation*}
Since $S_H\succeq\Lambda_H/2$, inversion reverses the Loewner order and
$\Lambda_H^{1/2}S_H^{-1}\Lambda_H^{1/2}\preceq2I_h$, proving (ii).

For part (iii), submultiplicativity and
$\norm{X_H}_{\op}^2/n=s_1$ give
\begin{align*}
\left\|
\Lambda_H^{1/2}\frac{X_H^\top\Xi X_H}{n}
\Lambda_H^{-1/2}
\right\|_{\op}
&\le
\sqrt{\frac{\lambda_1}{\lambda_h}}\,s_1\norm\Xi_{\op}
\le C_{C_H}\lambda_h\norm\Xi_{\op}.
\end{align*}
Furthermore,
\begin{equation*}
\tr(\Xi K_{H,\Lambda}^{(2)}\Xi)
\le\norm\Xi_{\op}^2\tr(K_{H,\Lambda}^{(2)}),
\qquad
\tr(K_{H,\Lambda}^{(2)})=\tr(\Lambda_HS_H)
\le C_{C_H}h\lambda_h^2.
\end{equation*}
Expanding $R=R_0+\Xi$ and using
$\norm{A+B}_F^2\le2\norm A_F^2+2\norm B_F^2$ proves
\eqref{eq:app-transfer-head-noise}.  Finally,
\begin{equation*}
\tr(RK_T^{(2)}R)
\le\norm{K_T^{(2)}}_{\op}\tr(R^2)
\le h\norm{K_T^{(2)}}_{\op}\norm R_{\op}^2,
\end{equation*}
which is \eqref{eq:app-transfer-tail-noise}.

For part (iv), the same singular-value decomposition yields
\begin{align*}
\frac{X_H^\top(K_H+c_0I_n)^{-1}X_H}{n}
&=V\diag\left\{\frac{s_j}{s_j+c_0}:1\le j\le h\right\}V^\top\\
&=S_H(S_H+c_0I_h)^{-1}.
\end{align*}
Congruence by $X_H/\sqrt n$ preserves Loewner order, proving
\eqref{eq:app-transfer-pcr-map}.  Since $M$ is symmetric positive
semidefinite and $\norm M_{\op}<1$,
\begin{align*}
\norm{(I_h-M)\beta_H^\star}_{\Lambda_H}^2
&\ge\lambda_h\norm{(I_h-M)\beta_H^\star}^2\\
&\ge\lambda_h(1-\norm M_{\op})^2\norm{\beta_H^\star}^2\\
&\ge\frac{\lambda_h}{\lambda_1}
\left(\frac{c_0}{s_1+c_0}\right)^2E_H,
\end{align*}
which proves \eqref{eq:app-transfer-pcr-bias}.
\end{proof}

\begin{proof}[Proof of Theorem 4.4 of the main paper]
Set $S_H=X_H^\top X_H/n$.  With probability at least $1-2e^{-t}$, Gaussian
singular-value concentration gives
\begin{equation}
\left\|\frac{Z_H^\top Z_H}{n}-I_h\right\|_{\op}
\le C\left\{\sqrt{\frac{h+t}{n}}+\frac{h+t}{n}\right\}.
\label{eq:app-gaussian-head-concentration}
\end{equation}
Indeed, for a fixed unit vector $v\in\R^h$,
$\norm{Z_Hv}^2$ is chi-square with $n$ degrees of freedom.  Its Bernstein
tail bound, followed by a union bound over a $1/4$-net of
$\mathbb S^{h-1}$ with cardinality at most $9^h$, gives
\eqref{eq:app-gaussian-head-concentration}.
When the right-hand side is at most a sufficiently small universal constant,
this implies
\begin{equation}
\frac12\Lambda_H
\preceq S_H\preceq
\frac32\Lambda_H,
\label{eq:app-head-comparison}
\end{equation}
and the nonzero eigenvalues of $K_H=X_HX_H^\top/n$ remain between fixed
multiples of $\lambda_h$ and $\lambda_1$.
On the event of Lemma~\ref{lem:weighted-tail-concentration}, let
\begin{equation*}
\delta=\norm{K_T^{(1)}-aI_n}_{\op},
\qquad
b=\norm{K_T^{(2)}}_{\op}.
\end{equation*}
Because $a=\lambda_h/\kappa$, the hypothesis
$q_1(t)\le c\min(\kappa,1)$ implies, after choosing $c$ sufficiently small,
\begin{equation}
\delta\le Ca q_1(t)\le c_0\min(\lambda_h,a),
\qquad
b\le b_T\{1+Cq_2(t)\}.
\label{eq:app-main-tail-event}
\end{equation}
Lemma~\ref{lem:random-head-transfer}(i) gives $s_h\ge\lambda_h/2$ and
$s_1/s_h\le3C_H$.  The cluster argument in
\eqref{eq:app-cluster-gap}--\eqref{eq:app-global-inverse-bound}, now applied
to this realized head, therefore gives
\begin{equation}
\norm{\widehat R}_{\op}\le\frac{C}{\lambda_h},
\qquad
\norm{\Xi}_{\op}
=\norm{\widehat R-K_H^+}_{\op}
\le C\frac{\delta}{\lambda_h^2},
\qquad
\widehat\mu_h-a\ge c\lambda_h.
\label{eq:app-main-inverse-event}
\end{equation}
The last inequality and $g_n\le c_g\lambda_h$ prove admissibility when
$c_g$ is small enough.

We next bound conditional prediction risk in the population metric.  The
derivation of \eqref{eq:operator-risk-identity} uses only block diagonality of
the population covariance, so with
\begin{equation*}
K_{H,\Lambda}^{(2)}=X_H\Lambda_HX_H^\top/n
\end{equation*}
it gives the same four-term identity with $D_H$ replaced by $\Lambda_H$.
The reference $R_0=K_H^+$ satisfies
$X_H^\top R_0X_H/n=I_h$, because $X_H$ has full column rank.  Thus, by
Lemma~\ref{lem:random-head-transfer}(iii) and
\eqref{eq:app-main-inverse-event}, the head bias is at most
\begin{equation}
C_{C_H}E_H\lambda_h^2\norm\Xi_{\op}^2
\le C_{C_H}E_H\frac{\delta^2}{\lambda_h^2}.
\label{eq:app-main-head-bias}
\end{equation}
For the noiseless tail term, Lemma~\ref{lem:random-head-transfer}(ii) and
\eqref{eq:app-main-inverse-event} give
\begin{equation}
\frac1n y_0^\top\widehat R K_T^{(2)}\widehat R y_0
\le C E_H\frac{b}{\lambda_h^2}.
\label{eq:app-main-tail-signal}
\end{equation}
The two variance terms are bounded by
\begin{align}
\frac{\sigma_\varepsilon^2}{n}
\tr(\widehat R K_{H,\Lambda}^{(2)}\widehat R)
&\le C V_H+C_{C_H}V_H\frac{\delta^2}{\lambda_h^2},
\label{eq:app-main-head-variance}\\
\frac{\sigma_\varepsilon^2}{n}
\tr(\widehat R K_T^{(2)}\widehat R)
&\le C V_H\frac{b}{\lambda_h^2},
\label{eq:app-main-tail-variance}
\end{align}
where the first line uses
Lemma~\ref{lem:random-head-transfer}(iii), and the second uses its rank-$h$
tail-noise bound.  Adding
\eqref{eq:app-main-head-bias}--\eqref{eq:app-main-tail-variance} yields
\begin{equation}
\Rbar_X\{\widehat\beta_{\dPCR}(a,h)\}
\le C_{C_H}\left[
V_H+(E_H+V_H)
\left\{\frac{\delta^2}{\lambda_h^2}
+\frac{b}{\lambda_h^2}\right\}
\right].
\label{eq:app-main-deterministic-upper}
\end{equation}
Finally,
\begin{equation*}
\frac{\delta^2}{\lambda_h^2}
\le C\frac{q_1(t)^2}{\kappa^2},
\qquad
\frac{b}{\lambda_h^2}
\le\frac{n}{d_1\kappa^2}\{1+Cq_2(t)\},
\end{equation*}
because $b_T=a^2n/d_1$.  This proves the dPCR bound in Theorem 4.4.

For ordinary PCR, let $G_r$ be its sample-space inverse.  The order
$0\preceq G_r\preceq K^{-1}$ holds for every $r$, including an arbitrary
choice of eigenvectors when a sample eigenvalue is repeated, because in a
spectral basis of $K$ each eigenvalue of $G_r$ is either $0$ or the
corresponding eigenvalue of $K^{-1}$.  From
$K_T^{(1)}\succeq(a-\delta)I_n$ and \eqref{eq:app-main-tail-event},
\begin{equation*}
0\preceq G_r\preceq K^{-1}
\preceq\{K_H+(a-\delta)I_n\}^{-1}.
\end{equation*}
Apply Lemma~\ref{lem:random-head-transfer}(iv) with $c_0=a-\delta$.
The noiseless head error alone is bounded below by
\begin{equation}
\frac1{C_H}
\left(\frac{a-\delta}{s_1+a-\delta}\right)^2E_H
\ge c_{C_H}\frac{E_H}{(1+C_H\kappa)^2},
\label{eq:app-random-head-pcr-lower}
\end{equation}
where $\delta\le a/2$ and $s_1\le3\lambda_1/2\le C_{C_H}\lambda_h$.
All tail-error and response-noise terms are nonnegative, so this proves the
rank-uniform PCR lower bound.  A union bound over the head event
($2e^{-t}$) and the two tail events ($4e^{-t}$) gives probability at least
$1-6e^{-t}$ for both conclusions simultaneously.
\end{proof}

\begin{proof}[Proof of Corollary 4.5 of the main paper]
Set $t_n=\log n$ and
\begin{equation*}
\eta_n=q_{1,n}(t_n)^2+\frac{n}{d_{1,n}}
\{1+q_{2,n}(t_n)\}.
\end{equation*}
The condition $h_n=o(n)$ gives $r_{H,n}(t_n)\to0$.  The matched-ratio bounds
and $q_{1,n}(t_n)\to0$ give
$q_{1,n}(t_n)\le c\min(\kappa_n,1)$ eventually, so all hypotheses of
Theorem 4.4 hold for sufficiently large $n$.  On its event, whose probability
is at least $1-6/n$, divide the dPCR upper bound by the PCR lower bound to get
\begin{align*}
\frac{\Rbar_X\{\widehat\beta_{\dPCR}(a_n,h_n)\}}
{\inf_r\Rbar_X\{\widehat\beta_{\PCR}(r)\}}
&\le C
\left[
\frac{V_{H,n}}{E_{H,n}}
+\left(1+\frac{V_{H,n}}{E_{H,n}}\right)\eta_n
\right],
\end{align*}
where the fixed shell and matched-ratio constants are absorbed into $C$.
This is precisely the asserted $O_{\Pp}$ bound because the displayed event
has probability tending to one.  Finally, the first summand of $q_1$ gives
\begin{equation*}
\frac{n}{d_{1,n}}
\le\frac{n}{n+t_n}q_{1,n}(t_n)^2\longrightarrow0.
\end{equation*}
Together with $q_{2,n}(t_n)=O(1)$ and
$V_{H,n}/E_{H,n}=\mathsf{SNR}_{H,n}^{-1}\to0$, this yields
$\eta_n\to0$ and proves convergence of the ratio to zero.
\end{proof}

\begin{proof}[Justification of Remark 4.6 of the main paper]
Let $d_i=\widehat\mu_i-a$.  On the event of
Proposition~\ref{prop:projected-inverse-risk}, $d_i\ge3s_h/4$ for $i\le h$.
If $|\widetilde a-a|\le s_h/4$, then
\begin{equation}
\left|
\frac1{\widehat\mu_i-\widetilde a}
-\frac1{\widehat\mu_i-a}
\right|
=\frac{|\widetilde a-a|}
{|d_i-\widetilde a+a|d_i}
\le C\frac{|\widetilde a-a|}{s_h^2}.
\label{eq:app-floor-resolvent}
\end{equation}
All retained eigenvectors are unchanged.  On the Gaussian head-concentration
event, $s_h\asymp\lambda_h$, which proves
the perturbation bound in Remark 4.6 of the main paper.  More explicitly, if
$D=\widehat R(\widetilde a)-\widehat R(a)$, then
$\operatorname{rank}(D)\le h$ and
\begin{equation*}
\norm D_{\op}\le
C\frac{|\widetilde a-a|}{\lambda_h^2},
\qquad
\norm{\widehat R(\widetilde a)}_{\op}\le\frac{C}{\lambda_h}.
\end{equation*}
In the population-metric proof of Theorem 4.4, replace
$\Xi=\widehat R(a)-R_0$ by $\Xi+D$.  The head-bias map in
\eqref{eq:app-transfer-bias}, followed by
$\norm{A+B}^2\le2\norm A^2+2\norm B^2$, adds at most
\begin{equation*}
C_{C_H}E_H\lambda_h^2\norm D_{\op}^2
\le C_{C_H}E_H
\left(\frac{\widetilde a-a}{\lambda_h}\right)^2.
\end{equation*}
The same expansion in \eqref{eq:app-transfer-head-noise} adds the
corresponding term with $V_H$.  The two tail terms retain their previous
order because $\widehat R(\widetilde a)$ has rank $h$ and operator norm
$O(\lambda_h^{-1})$.  Absorbing the harmless factor multiplying the oracle
bound gives precisely the additional
$C_{C_H}(E_H+V_H)\{(\widetilde a-a)/\lambda_h\}^2$ stated in the remark.
\end{proof}

\begin{lemma}[Containment of the residual sample cluster]
\label{lem:residual-cluster-containment}
Let $A,B\in\R^{n\times n}$ be symmetric, with $A\succeq0$ and
$\operatorname{rank}(A)\le h<n$.  If $\norm{B-aI_n}_{\op}\le\delta$, then
\begin{equation}
a-\delta\le\operatorname{eig}_i(A+B)\le a+\delta,
\qquad h<i\le n.
\label{eq:app-residual-cluster}
\end{equation}
If $\operatorname{rank}(A)=h$, then also
\begin{equation}
\operatorname{eig}_h(A+B)
\ge\operatorname{eig}_h(A)+a-\delta.
\label{eq:app-head-cluster-lower}
\end{equation}
Consequently, every statistic in the convex hull of the eigenvalues below
rank $h$ differs from $a$ by at most $\delta$.
\end{lemma}

\begin{proof}
The operator inequality $B\succeq(a-\delta)I_n$ and $A\succeq0$ gives
$A+B\succeq(a-\delta)I_n$, proving the lower bound in
\eqref{eq:app-residual-cluster}.  Since $\dim\ker(A)\ge n-h$, the Rayleigh
quotient of $A+B$ on $\ker(A)$ equals that of $B$ and is at most
$a+\delta$.  The Courant--Fischer formula therefore gives
$\operatorname{eig}_{h+1}(A+B)\le a+\delta$; monotonicity in the eigenvalue
index proves the upper bound for every $i>h$.  Finally,
$A+B\succeq A+(a-\delta)I_n$, so eigenvalue monotonicity proves
\eqref{eq:app-head-cluster-lower}.  The convex-hull conclusion is immediate.
\end{proof}

\begin{proof}[Proof of Corollary 4.7 of the main paper]
Apply Lemma~\ref{lem:residual-cluster-containment} with
$A=K_H$ and $B=K_T^{(1)}$.  It gives, deterministically on the event
$\norm{K_T^{(1)}-aI_n}_{\op}\le\delta$,
\begin{equation}
|\widehat\mu_i-a|\le\delta\quad(i>h),
\qquad
|\widehat a_h-a|\le\delta.
\label{eq:app-weyl-containment}
\end{equation}
This is a pathwise statement and therefore does not require the floor
estimate to be independent of the retained eigenvalues or empirical
eigenvectors.  On the event of
Lemma~\ref{lem:weighted-tail-concentration}, $\delta\le Caq_1(t)$, and the
smallness condition \eqref{eq:sharp-smallness} gives
$\delta\le s_h/4$.  Thus Remark 4.6 of the main paper applies with
$\widetilde a=\widehat a_h$ and adds at most
$C_{C_H}(E_H+V_H)\delta^2/s_h^2
\le C_{C_H}(E_H+V_H)q_1(t)^2/\kappa^2$, which is absorbed by the first brace
term of \eqref{eq:sharp-dpcr-bound}.  Admissibility follows from
\begin{equation*}
\widehat\mu_h-\widehat a_h
\ge(s_h+a-\delta)-(a+\delta)
=s_h-2\delta\ge\frac{s_h}2,
\end{equation*}
where the lower bound for $\widehat\mu_h$ is
\eqref{eq:app-head-cluster-lower}.  For the Gaussian random-design head, the sample head Gram is
positive semidefinite of rank $h$ almost surely, so
\eqref{eq:app-weyl-containment} is unchanged; on the head-concentration event
used in Theorem 4.4 of the main paper, $s_h\ge\lambda_h/2$.  The explicit
operator perturbation in the proof of Remark 4.6 therefore gives
\begin{equation*}
\norm{\widehat R(\widehat a_h)-\widehat R(a)}_{\op}
\le C\frac{\delta}{\lambda_h^2}.
\end{equation*}
Equations \eqref{eq:app-main-head-bias}--\eqref{eq:app-main-tail-variance},
with this additional operator difference, add at most
$C_{C_H}(E_H+V_H)\delta^2/\lambda_h^2$ and leave the two tail terms at their
previous order.  Since
$\delta^2/\lambda_h^2\le Cq_1(t)^2/\kappa^2$, this term is absorbed in the
oracle upper bound.  The PCR lower bounds do not involve the correction
level, so the replacement in Corollary 4.5 of the main paper follows on the
same event.
\end{proof}

\begin{proof}[Proof of Corollary 4.8 of the main paper]
Set $t=t_n=\log n$ and write $s_T=X_T\beta_T^\star$.  If $E_T=0$, then
$s_T=0$ almost surely and the exact-alignment argument already applies, so
assume $E_T>0$.  Because $\beta_T^\star$ is deterministic (or independent of
the design) and the columns of $X_T$ are independent,
\begin{equation*}
s_T=\sum_{j>h}\sqrt{\lambda_j}\,\beta_j^\star z_j
\sim N(0,E_TI_n).
\end{equation*}
The chi-square concentration inequality therefore gives
\begin{equation}
\frac{\norm{s_T}^2}{n}
\le E_T\left\{1+C\sqrt{\frac{t}{n}}+C\frac{t}{n}\right\}
\le C E_T
\label{eq:app-tail-response-concentration}
\end{equation}
with probability at least $1-2e^{-t}$ for all sufficiently large $n$.  This
event need not be independent of the head or weighted-tail events; a union
bound is sufficient below.

First consider dPCR.  For either $\widetilde a=a$ or
$\widetilde a=\widehat a_h$, let $R=\widehat R(\widetilde a)$ be its retained
sample-space operator.  On the events used in
Theorem 4.4 and Corollary 4.7 of the main paper,
$\norm R_{\op}\le C/\lambda_h$.  Define
$A_R\beta_T^\star=X^\top Rs_T/n$ and
$K_H^{(2)}=X_H\Lambda_HX_H^\top/n$.  Then
\begin{align}
\norm{A_R\beta_T^\star}_{\Sigma}^2
&=\frac1n s_T^\top R\{K_H^{(2)}+K_T^{(2)}\}Rs_T
\nonumber\\
&\le
C E_T\left[1+\frac{n}{d_1\kappa^2}\{1+q_2(t)\}\right].
\label{eq:app-tail-signal-fitted-risk}
\end{align}
Indeed, head concentration gives
$\norm{K_H^{(2)}}_{\op}\le C_{C_H}\lambda_h^2$, while
Lemma~\ref{lem:weighted-tail-concentration} and
\eqref{eq:app-random-substitution} control $K_T^{(2)}$.
By linearity, the full-signal estimator equals the estimator formed from
$X_H\beta_H^\star+\varepsilon$ plus $A_R\beta_T^\star$.  Hence
$\norm{u+v}_{\Sigma}^2\le2\norm u_{\Sigma}^2+2\norm v_{\Sigma}^2$, with
$u$ the head-only estimation error and
$v=A_R\beta_T^\star-(0,\beta_T^\star)$.  Since
$\norm v_{\Sigma}^2\le
2\norm{A_R\beta_T^\star}_{\Sigma}^2+2E_T$,
\eqref{eq:app-tail-signal-fitted-risk} shows that the exact-alignment dPCR
upper bound gains at most
\begin{equation}
C E_T\left[1+\frac{n}{d_1\kappa^2}\{1+q_2(t)\}\right].
\label{eq:app-tail-signal-dpcr-cost}
\end{equation}

It remains to check that small tail signal cannot remove the all-rank PCR
barrier.  Let $G_r$ be the PCR inverse from the proof of
Proposition~\ref{thm:sharp-pcr-lower}, let $u_r=(M_r-I_h)\beta_H^\star$ be its head-only
bias, and put $w_r=X_H^\top G_rs_T/n$.  Since
$K\succeq(a-\delta)I_n$ and $\delta\le a/2$ on the sharp-floor event,
$\norm{G_r}_{\op}\le2/a$ uniformly in $r$.  Therefore
\begin{equation}
\norm{w_r}_{\Lambda_H}^2
=\frac1n s_T^\top G_rK_H^{(2)}G_rs_T
\le C\kappa^2E_T.
\label{eq:app-tail-signal-pcr-head}
\end{equation}
The inequality $\norm{u+w}^2\ge\frac12\norm u^2-\norm w^2$ and the
head-only lower bound in \eqref{eq:app-random-head-pcr-lower} now yield,
uniformly over $r$,
\begin{equation}
\Rbar_X\{\widehat\beta_{\PCR}(r)\}
\ge c\frac{E_H}{(1+C\kappa)^2}-C\kappa^2E_T.
\label{eq:app-approx-pcr-lower}
\end{equation}

The matched-ratio assumptions make $\kappa_n$ bounded above and below.  Since
$\rho_{T,n}\to0$, \eqref{eq:app-approx-pcr-lower} is at least
$c'E_{H,n}$ eventually, uniformly over $r$.  On the intersection of the
Theorem 4.4 event, the plug-in containment event when needed, and
\eqref{eq:app-tail-response-concentration}, whose probability is at least
$1-8/n$, division gives
\begin{equation*}
\begin{split}
\frac{\Rbar_X\{\widehat\beta_{\dPCR}(\widetilde a_n,h_n)\}}
{\inf_r\Rbar_X\{\widehat\beta_{\PCR}(r)\}}
\le{}& C\left[
\mathsf{SNR}_{H,n}^{-1}
+\{1+\mathsf{SNR}_{H,n}^{-1}\}\eta_n
\right]\\
&+C\rho_{T,n}\left\{1+
\frac{n}{d_{1,n}\kappa_n^2}(1+q_{2,n}(t_n))\right\}
.
\end{split}
\end{equation*}
Corollary 4.5 gives
$n\{1+q_{2,n}(t_n)\}/d_{1,n}=o(1)$, so the final brace is $1+o(1)$.
This proves the asserted $O_{\Pp}(\rho_{T,n})$ addition and the displayed
ratio bound for both the oracle and plug-in floors.
\end{proof}

\begin{proof}[Justification of Remark 4.9 of the main paper]
By \eqref{eq:app-weyl-containment}, at most $h\le n/4$ sample eigenvalues lie
outside $[a-\delta,a+\delta]$, so the median $\widehat a_0$ of all $n$ sample
eigenvalues lies in $[a-\delta,a+\delta]$.  Write
$T=(1+\underline\kappa/2)\widehat a_0$ and suppose
$\delta/a\le \underline\kappa/(8+2\underline\kappa)$, which
Lemma~\ref{lem:weighted-tail-concentration} guarantees once
$q_1(t)\le c(\underline\kappa)$.  Then
\begin{equation*}
T\ge(a-\delta)\left(1+\frac{\underline\kappa}2\right)\ge a+\delta
\ge\max_{i>h}\widehat\mu_i,
\qquad
T\le(a+\delta)\left(1+\frac{\underline\kappa}2\right)
<a+s_h-\delta\le\widehat\mu_h,
\end{equation*}
where the second chain uses $s_h\ge\underline\kappa a$.  Hence
exactly the $h$ head-cluster eigenvalues exceed $T$, so $\widehat h=h$ and
Corollary 4.7 of the main paper applies at $(\widehat a_{\widehat h},\widehat h)$.
For the Gaussian random-design head, the concentration event in
\eqref{eq:app-gaussian-head-concentration} can be chosen small enough that
$s_h\ge\lambda_h/2$.  Therefore
$s_h/a\ge c_\kappa/2$, which justifies the stated choice of
$\underline\kappa$.
\end{proof}

\begin{proof}[Proof of Corollary 4.11 of the main paper]
Under the flat-tail specialization in Corollary 4.11 of the main paper, direct
substitution gives
$d_1=d_2=m$ and
\begin{equation}
q_1(t)=q_2(t)
=\sqrt{\frac{n+t}{m}}+\frac{n+t}{m}.
\label{eq:app-flat-q}
\end{equation}
Also,
$K_T^{(1)}/a=Z_TZ_T^\top/m$.  The standard Gaussian singular-value bound
therefore gives
\[
\norm{K_T^{(1)}/a-I_n}_{\op}
=O_{\Pp}\!\left(\sqrt{n/m}+n/m\right),
\]
which is the Wishart-collapse display in Corollary 4.11 of the main paper.
For $t=\log n$ and $m/n\to\infty$, both
$q_1(t)^2$ and $n\{1+q_2(t)\}/d_1$ are $O(n/m)$.  The risk and risk-ratio
displays in Corollary 4.11 of the main paper now follow
by substitution into the two bounds of Theorem 4.4 of the main paper.  The plug-in
conclusion follows from Corollary 4.7 of the main paper, and the SNR condition
makes the ratio vanish.
\end{proof}

\subsection{Fixed-aspect calculations}

\begin{proof}[Proof of Theorem 5.2 of the main paper]
Let $\widehat v$ be the top sample eigenvector, oriented so that its first
coordinate is nonnegative.  Standard spiked-covariance asymptotics give
\begin{equation*}
\widehat\mu_1\to\rho,
\qquad
|\widehat v_1|^2\to\chi.
\end{equation*}
Because $g<\rho-a_0$, the first convergence also gives
$\widehat\mu_1-a_0\ge g$ with probability tending to one, proving
admissibility of $(a_0,1)$.
Ignoring response noise for the moment, define the empirical multiplier
\begin{equation*}
c_n(a_0)=\frac{\widehat\mu_1}{\widehat\mu_1-a_0}.
\end{equation*}
The noiseless rank-$1$ dPCR component is
$c_n(a_0)\widehat v\widehat v^\top\beta^\star$.  Its head coordinate is
$c_n(a_0)\widehat v_1^2b$, and the squared Euclidean norm of its tail
coordinates is
$c_n(a_0)^2\widehat v_1^2(1-\widehat v_1^2)b^2$.  Because
$\widehat\mu_1\xrightarrow{\Pp}\rho>a_0$, the continuous mapping theorem gives
$c_n(a_0)\xrightarrow{\Pp}c(a_0)$.  Combining this convergence with
$\widehat v_1^2\xrightarrow{\Pp}\chi$ and weighting the head and tail blocks
by $\lambda$ and $\tau$ proves the dPCR signal-risk formula in Theorem 5.2 of
the main paper; setting $a_0=0$ gives $c_n(0)=1$ and proves the corresponding
PCR formula.

The response-noise coefficient along $\widehat v$ has variance
\begin{equation*}
\frac{\sigma_\varepsilon^2\widehat\mu_1}
{n(\widehat\mu_1-a_0)^2},
\end{equation*}
and its conditional prediction contribution is this variance multiplied by
\begin{equation*}
\widehat v^\top\Sigma\widehat v
=\lambda\widehat v_1^2+\tau(1-\widehat v_1^2).
\end{equation*}
Admissibility bounds the denominator away from zero with probability tending
to one, while $\widehat\mu_1$ and the last display are tight.  The
response-noise term is therefore $O_{\Pp}(n^{-1})$.

It remains to minimize
\begin{equation*}
r(c)=\lambda(c\chi-1)^2+\tau c^2\chi(1-\chi).
\end{equation*}
Differentiation gives
\begin{equation*}
c^\star=\frac{\lambda}{\lambda\chi+\tau(1-\chi)}.
\end{equation*}
The outlier and overlap formulas further imply
\begin{equation*}
\lambda\chi+\tau(1-\chi)=\frac{\lambda^2}{\rho}.
\end{equation*}
Hence $c^\star=\rho/\lambda$ and solving
$c^\star=\rho/(\rho-a_0^\star)$ gives
$a_0^\star=\rho-\lambda=\gamma\tau\lambda/(\lambda-\tau)$.  The first form
displayed in Theorem 5.2 of the main paper is the same quantity: since
$(1-\chi)(\lambda-\tau)
=\lambda-\{\lambda\chi+\tau(1-\chi)\}
=\lambda-\lambda^2/\rho
=\lambda(\rho-\lambda)/\rho$,
one has
$\rho(1-\chi)(1-\tau/\lambda)
=\rho(1-\chi)(\lambda-\tau)/\lambda
=\rho-\lambda$.  Completing
the square gives
\begin{equation*}
r(c)=r(c^\star)
+\chi\{\lambda\chi+\tau(1-\chi)\}(c-c^\star)^2,
\end{equation*}
and direct substitution yields
\begin{equation*}
r(c^\star)
=\frac{\lambda\tau(1-\chi)}
{\lambda\chi+\tau(1-\chi)}
=\frac{\gamma\lambda\tau^2}{(\lambda-\tau)^2}.
\end{equation*}
Above the transition, $0<\chi<1$ and $\lambda>\tau$, hence
$\lambda\chi+\tau(1-\chi)<\lambda$ and $c^\star>1$.  The strict convexity
display then gives $r(c^\star)<r(1)$ and $r(c^\star)<r(0)=\lambda$.
Multiplication by $b^2>0$ proves both strict-improvement claims.
\end{proof}

\begin{proof}[Verification of the normalized fixed-aspect formulas]
Divide $\rho$, $a_0$, and every population eigenvalue by $\tau$.  The
rank-$1$ signal risk from Theorem 5.2 becomes
\begin{equation*}
b^2\tau\left[
\xi\{c(a_0)\chi_\gamma-1\}^2
+c(a_0)^2\chi_\gamma(1-\chi_\gamma)
\right],
\qquad
c(a_0)=\frac{\rho_\gamma^\circ}
{\rho_\gamma^\circ-\alpha_0},
\end{equation*}
so the normalized expression depends only on
$(\xi,\gamma,\alpha_0)$.  With a scaled denominator margin
$g_n=\bar g\tau_n$, the corrected denominator is bounded below by a fixed
multiple of $\tau_n$.  The conditional response-noise variance along the
retained component is then $O_{\Pp}\{\sigma_{\varepsilon,n}^2/
(n\tau_n)\}$, while its prediction weight is $O_{\Pp}(\tau_n)$; hence its
prediction contribution is $O_{\Pp}(\sigma_{\varepsilon,n}^2/n)$.

For the collapsing-bulk calculation, let
$\kappa_n=\xi_n/\gamma_n\to\kappa\in(0,\infty)$.  Directly from the outlier
formula,
\begin{equation*}
\frac{\rho_{\gamma_n}}{a_n}
=\frac{\lambda_n}{\gamma_n\tau_n}
+\frac{\lambda_n}{\lambda_n-\tau_n}
=\frac{\xi_n}{\gamma_n}
+\frac{\xi_n}{\xi_n-1}
\longrightarrow\kappa+1.
\end{equation*}
Similarly,
\begin{equation*}
\chi_{\gamma_n}
=\frac{1-\gamma_n/(\xi_n-1)^2}
{1+\gamma_n/(\xi_n-1)}
\longrightarrow\frac{\kappa}{1+\kappa}.
\end{equation*}
Finally, the optimal-correction identity from Theorem 5.2 gives
\begin{equation*}
\frac{a_{0,n}^\star}{a_n}
=\frac{\lambda_n}{\lambda_n-\tau_n}
=\frac{\xi_n}{\xi_n-1}\longrightarrow1.
\end{equation*}
The optimal corrected denominator is
$\rho/c^\star=\lambda$ by $c^\star=\rho/\lambda$.  These calculations prove
all normalization and pointwise large-aspect displays in the main paper.
\end{proof}

\begin{proof}[Proof of Proposition 5.3 of the main paper]
In the notation of Green and Romanov (2025), the empirical population
spectral distribution and signal-weighted population measure are
\begin{equation*}
H_n=\frac{\delta_\lambda+m\delta_\tau}{m+1}
\Longrightarrow\delta_\tau,
\qquad
G_n=\delta_\lambda\Longrightarrow\delta_\lambda.
\end{equation*}
The limit $H=\delta_\tau$ is compactly supported and bounded away from zero,
the single spike $\lambda$ remains fixed, Gaussian entries satisfy the moment
assumptions, and $(m+1)/n\to\gamma>1$.  The empirical measure denoted
$\widehat B_n$ in Section 5.1 of that paper is, after setting its deterministic
signal to $e_1$, exactly
\begin{equation*}
\sum_{i=1}^p|\langle\widehat v_i,e_1\rangle|^2
\delta_{\widehat\mu_i}
=\nu_{\lambda,n}^{\mathrm{full}}.
\end{equation*}
Its Lemma 1 gives pointwise convergence of the Stieltjes transform of this
measure, and the weak-convergence argument immediately following that lemma
yields a probability measure $\nu_\lambda^{\mathrm{full}}$ such that, for
every bounded continuous $\varphi$,
\begin{equation}
e_1^\top\varphi(\widehat\Sigma)e_1
\xrightarrow{\Pp}
\int\varphi(x)\,d\nu_\lambda^{\mathrm{full}}(x).
\label{eq:app-gr-weak-limit}
\end{equation}
Green and Romanov normalize their full sample ESD by $p$, consistently with
$p/n\to\gamma$.  The signal measure above is not ESD-normalized: its weights
sum to $\norm{e_1}^2=1$, which is exactly our normalization.

We derive the null atom directly.  Put
$A=X_TX_T^\top=\tau Z_TZ_T^\top$ and $x=\sqrt\lambda z$, so
$XX^\top=A+xx^\top$.  Since $m/n\to\gamma>1$, $A$ is invertible almost surely.
The null atom is the squared length of the projection of $e_1$ onto
$\ker(X)$:
\begin{align*}
\nu_{\lambda,n}^{\mathrm{full}}(\{0\})
&=1-e_1^\top X^\top(XX^\top)^{-1}Xe_1\\
&=1-x^\top(A+xx^\top)^{-1}x\\
&=\frac{1}{1+x^\top A^{-1}x},
\end{align*}
where the last equality is the Sherman--Morrison identity.  Conditional on
$Z_T$, the variance of the centered quadratic form below is
$2\tr\{(Z_TZ_T^\top)^{-2}\}$.  The lower Marchenko--Pastur edge is positive
for $\gamma>1$, so this trace is $O_{\Pp}(n^{-1})$.  Conditional Chebyshev
therefore gives
\begin{equation*}
z^\top(Z_TZ_T^\top)^{-1}z
-\tr\{(Z_TZ_T^\top)^{-1}\}\xrightarrow{\Pp}0.
\end{equation*}
The Marchenko--Pastur inverse moment, equivalently the inverse-Wishart trace
law, gives
\begin{equation*}
\tr\{(Z_TZ_T^\top)^{-1}\}\xrightarrow{\Pp}\frac1{\gamma-1}.
\end{equation*}
Consequently,
\begin{equation*}
\nu_{\lambda,n}^{\mathrm{full}}(\{0\})
\xrightarrow{\Pp}
\left\{1+\frac{\lambda}{\tau(\gamma-1)}\right\}^{-1}.
\end{equation*}
The weak limit has this atom because zero is separated from the positive
sample spectrum with probability tending to one.  Green and Romanov's
support-and-atom characterization (their Lemma 3) places the positive mass on
the Marchenko--Pastur bulk and, above the transition, assigns mass $\chi$ to
the outlier at $\rho$.  Restriction to $(0,\infty)$ therefore defines
$\nu_\lambda$ and proves all claims of Proposition 5.3.
\end{proof}

\begin{lemma}[Convergence for a fixed spectral cutoff]
\label{lem:fixed-cutoff-convergence}
Let $\zeta_n$ be finite positive measures supported, with probability tending
to one, in a common compact subset of $[0,\infty)$, and suppose
$\zeta_n\Rightarrow\zeta$ in probability.  Fix $s>0$, $g>0$, and
$a_0\in[0,s-g]$, and assume $\zeta(\{s\})=0$.  With
\begin{equation*}
f_{a_0,s}(x)=\frac{x}{x-a_0}\1\{x\ge s\},
\end{equation*}
one has
\begin{equation}
\int f_{a_0,s}\,d\zeta_n\xrightarrow{\Pp}
\int f_{a_0,s}\,d\zeta,
\qquad
\int f_{a_0,s}^2\,d\zeta_n\xrightarrow{\Pp}
\int f_{a_0,s}^2\,d\zeta.
\label{eq:app-fixed-filter-convergence}
\end{equation}
The same conclusion holds for
$x(x-a_0)^{-2}\1\{x\ge s\}$.
\end{lemma}

\begin{proof}
On $[s,\infty)$, the denominator is at least $g$, so all three functions in
the statement are uniformly bounded on the common compact support.  Their
only possible discontinuity on that support is at $s$.  Approximate
$\1\{x\ge s\}$ from above and below by continuous functions that differ only
on $(s-\epsilon,s+\epsilon)$, apply weak convergence, and then let
$\epsilon\downarrow0$ using $\zeta(\{s\})=0$.  This proves each convergence
in \eqref{eq:app-fixed-filter-convergence}; the same approximation applies to
the third function.
\end{proof}

\begin{proof}[Derivation of Theorem 5.4 of the main paper]
Let $P_s$ be the empirical spectral projector onto nonzero eigenvalues at least
$s$.  The empirical spike spectral measure is
\[
\nu_{\lambda,n}
=\sum_{i=1}^n|\ip{\widehat v_i}{e_1}|^2
\delta_{\widehat\mu_i}.
\]
Proposition 5.3 of the main paper gives convergence of the full signal
spectral measure.  Because $\gamma>1$, its null atom is separated from the
positive Marchenko--Pastur support; restricting to $(0,\infty)$ gives
$\nu_{\lambda,n}\Rightarrow\nu_\lambda$.  Apply
Lemma~\ref{lem:fixed-cutoff-convergence} with $\zeta_n=\nu_{\lambda,n}$ and
$f=f_{a_0,s}$ to obtain
\begin{equation*}
e_1^\top f(\widehat\Sigma)e_1\to C(a_0,s),
\qquad
e_1^\top f(\widehat\Sigma)^2e_1\to B(a_0,s).
\end{equation*}
Therefore the fitted signal has head coefficient $bC$ and squared tail norm
$b^2(B-C^2)$.  This gives the first line of
the fixed-aspect risk formula in Theorem 5.4 of the main paper.

For response noise, conditional on $X$, the coefficient of $\widehat v_i$
has variance
$\sigma_\varepsilon^2\widehat\mu_i/
\{n(\widehat\mu_i-a_0)^2\}$.  Orthogonality of the left singular vectors
eliminates cross-covariances, so the exact conditional contribution is
\begin{equation}
\frac{\sigma_\varepsilon^2}{n}
\sum_{\widehat\mu_i\ge s}
\frac{\widehat\mu_i}{(\widehat\mu_i-a_0)^2}
\widehat v_i^\top\Sigma\widehat v_i.
\label{eq:app-fixed-noise-sum}
\end{equation}
Since $\Sigma=\tau I+(\lambda-\tau)e_1e_1^\top$, first consider its isotropic
part.  The positive-spectrum empirical measure
$n^{-1}\sum_{i=1}^n\delta_{\widehat\mu_i}$ converges to
$F_{\gamma,\tau}$, and the third conclusion of
Lemma~\ref{lem:fixed-cutoff-convergence} gives convergence of that part to
\[
\sigma_\varepsilon^2\tau
\int_s^\infty\frac{x}{(x-a_0)^2}\,dF_{\gamma,\tau}(x).
\]
For the remaining spike part, the multiplier in
\eqref{eq:app-fixed-noise-sum} is uniformly bounded because
$\widehat\mu_i-a_0\ge g$ on the retained spectrum and spectral confinement
bounds $\widehat\mu_i$.  Hence its absolute value is at most
\begin{equation*}
\frac{C\sigma_\varepsilon^2|\lambda-\tau|}{n}
\sum_{i=1}^n|\ip{\widehat v_i}{e_1}|^2
\le\frac{C\sigma_\varepsilon^2|\lambda-\tau|}{n}=o(1).
\end{equation*}
This proves the second line.  Setting $a_0=0$ makes $f$ a projector,
so $B=A$ and the ordinary-PCR specialization in Theorem 5.4 of the main paper
follows.

The null-mass identity in Proposition 5.3 of the main paper gives
$A(s)\le1-\nu_\lambda^{\mathrm{full}}(\{0\})<1$ uniformly on the extended
cutoff domain $\mathcal S_\gamma$.  Hence the signal term alone is bounded
below by
$b^2\lambda\{\nu_\lambda^{\mathrm{full}}(\{0\})\}^2$, proving
$L_{\PCR}>0$.  This argument is pointwise in the limiting cutoff; it does not
supply the proposed uniform finite-sample interchange stated after Theorem 5.4
of the main paper.
\end{proof}

\begin{proof}[Verification of the fixed-aspect decomposition]
If $A=0$, then $r_A(c)=b^2\lambda$ for every $c$, while
the optimal-scale risk equals $b^2\lambda$ and the unit-multiplier excess
equals zero.  Thus the decomposition holds
at this endpoint, although the minimizer is not unique.  Assume henceforth
$A>0$.
For a PCR projector with spike mass $A$, consider a hypothetical scalar
multiplier $c$ on that same projector.  Its signal risk is
\begin{equation*}
r_A(c)=b^2\{\lambda(cA-1)^2+\tau c^2(A-A^2)\}
=b^2\{D(A)c^2-2\lambda Ac+\lambda\}.
\end{equation*}
The minimizer is $c_A^\star=\lambda A/D(A)$.  Completing the square gives
\begin{equation*}
r_A(1)=r_A(c_A^\star)+b^2D(A)(1-c_A^\star)^2.
\end{equation*}
Direct simplification of the two terms yields
the displayed formulas for the optimal-scale risk and unit-multiplier excess in
the main paper.
\end{proof}

\clearpage

\bibliographystyle{plainnat}
\bibliography{defloored_pcr}

\end{document}